\documentclass[a4paper,10pt]{article}

\usepackage[utf8]{inputenc}
\usepackage[T1]{fontenc}

\usepackage[english]{babel}

\usepackage[inline]{enumitem} 

\usepackage{authblk} 

\usepackage{cancel}
\usepackage{caption}
\usepackage{subcaption}
\usepackage{tabularx} 

\usepackage{geometry}

\usepackage{amsmath,amssymb,amsfonts,amsthm,amsbsy,stmaryrd,dsfont,bm}
\usepackage{bbold} 
\usepackage{nccmath} 
\usepackage{units}   

\usepackage{graphicx}
\usepackage{xcolor}

\usepackage{hyperref}

\usepackage[capitalise]{cleveref} 

\newlist{steps}{enumerate}{1}
\setlist[steps, 1]{wide=0pt, label=Step \arabic*., font=\scshape}
\newcommand{\jsize}{\xi} 
\newcommand{\eps}{\varepsilon}
\DeclareMathOperator{\sgn}{sgn}
\DeclareMathOperator{\Cov}{Cov}
\newcommand{\dd}{\mathop{}\!\mathrm{d}}
\newcommand{\abs}[1]{\left\lvert#1\right\rvert}
\newcommand{\absd}[1]{\left\|#1\right\|}
\renewcommand{\textbf}[1]{\begingroup\bfseries\mathversion{bold}#1\endgroup} 
\newcommand{\citer}[2][]{#1 in \cite{#2}} 
\newcommand{\comment}[1]{}

\theoremstyle {definition} 

\theoremstyle{plain} \newtheorem {theo} {Theorem}[section]

\newtheorem{coro}[theo]{Corollary}
\newtheorem{prop}[theo]{Proposition}
\newtheorem{lem}[theo]{Lemma}
\newtheorem{remark}[theo]{Remark}

\begin{document}
	
	\title{Scaling limit of a kinetic inhomogeneous stochastic system in the quadratic potential}
	\author{
		Thomas Cavallazzi and Emeline Luirard
	}
	\affil{\small Univ Rennes, CNRS, IRMAR - UMR 6625, F-35000 Rennes, France\\
		\texttt{\small \{Thomas.Cavallazzi,Emeline.Luirard\}@univ-rennes1.fr}}
	\date{%
	}
	\maketitle
	
	{\small\noindent {\bf Abstract:}~We consider a particle evolving in the quadratic potential and subject to a time-inhomogeneous frictional force and to a random force. The couple of its velocity and position is solution to a stochastic differential equation driven by a symmetric $\alpha$-stable Lévy process with $\alpha \in (1,2]$ and the frictional force is of the form $t^{-\beta}\sgn(v)|v|^\gamma$. We identify three regimes for the behavior in long-time of the couple velocity-position with a suitable rescaling, depending on the balance between the frictional force and the index of stability $\alpha$ of the noise.}
	{\small\noindent{\bf MSC2020 Subject Classification:}~Primary~60H10; Secondary~60F17; 60G52; 60G18; 60J65; 34E05.}
\\

\section{Introduction and main results}

\subsection{Model and motivations}

In this paper, we study the long-time behavior of a stochastic system modelling a particle, with velocity $V\in \mathbb{R}$ and position $X\in \mathbb{R}$. The particle evolves in the quadratic potential $\mathcal{U}:x\mapsto\frac{x^2}{2}$, and is subject to a time-inhomogeneous frictional force $b$ and to a random force $L$. The dynamics of the particle is described by the following stochastic damping Hamiltonian system driven by an $\alpha$-stable process $L$ with $\alpha\in(1,2]$
\begin{equation}\label{SDE}
	\begin{cases}
		\dd V_t= \dd L_t -b(t,V_t)\dd t - \nabla \mathcal{U}(X_t) \dd t, \\
		\dd X_t= V_t\dd t,                                   \\ (V_{t_0},X_{t_0})=(v_{0},x_{0}), \quad t_0>0 \text{ being fixed.}
	\end{cases}
\end{equation}

The driving process $L$ models a random force coming from the interaction of the particle with its environment represented by a surrounding heat bath. As in the classical Langevin model (see \cite{Langevintheoriemouvementbrownien1908}), $L$ can be a Brownian motion denoted by $B$ in the sequel. It corresponds to take $\alpha =2$. It is natural to consider other types of noises such as Lévy processes, which are also largely used to model physical and biological systems (Lévy flights and anomalous diffusion), see e.g. \cite{Mann_fractal_fractional_diffusions} for the physical point of view, \cite{DitlevsenObservationastablenoise1999} in stochastic climate dynamics, and \cite{Jourdain_fractional_diffusion} for the mathematical point of view. The case where $L$ is an $\alpha$-stable process is of particular interest. It is a generalization of the Brownian motion with jumps since it satisfies that for any $c>0$, $(c^{\frac{1}{\alpha}}L_{t/c})_t$ has the same distribution as $L$ (self-similarity property).\\

\medskip
Degenerate systems like \eqref{SDE} have been intensively studied for several years. In particular, the existence and uniqueness of solutions to degenerate SDEs have been discussed in many works. These models are called degenerate because the noise is only present in one component of the system but can be transferred into others by drift terms. The well-posedness of these systems, when their deterministic version is ill-posed, can be proved by taking advantage of the regularizing effect of the noise and of its propagation through the whole system. The case of Brownian degenerate SDEs has been of course wildly explored, see e.g. \cite{FedrizziRegularitystochastickinetic2017},  \cite{WangDegenerateSDEsHilbert2015}, \cite{ZhangStochasticHamiltonianflows2016}, \cite{RaynalStrongexistenceuniqueness2017}, \cite{HonoreStrongregularizationBrownian2018} and references therein. The time-dependence is treated in the last four cited papers. The case of a Lévy driving process is more recent, see e.g. \cite{ZhangDensitiesSDEsDriven2014} in a time-homogeneous setting, and \cite{MarinoWeakwellposednessdegenerate2021} for drifts depending on time.\\

From another point of view, stochastic Hamiltonian systems, as \eqref{SDE} with $b=0$, have been widely studied. An interesting problem is to understand their asymptotic behaviors. The Hamiltonian process associated with this system is defined, for $t\geq t_0$, by $H_t:= \frac{1}{2}|V_t|^2 + \mathcal{U}(X_t)$. For example, the long-time dynamics of the Hamiltonian process under a suitable rescaling is studied in \cite{AlbeverioLongtimebehavior1994}. The case of time-homogeneous damping Hamiltonian systems is tackled in \cite{WuLargemoderatedeviations2001} (see also references therein).\\

The long-time behavior of a particle evolving in a free potential, i.e.\ $\mathcal{U} = 0$, has already been studied, see e.g.\ \cite{GradinaruExistenceasymptoticbehaviour2013}, \cite{FournierOnedimensionalcritical2021}, \cite{GradinaruAsymptoticbehaviortimeinhomogeneous2021a},
\cite{GradinaruAsymptoticbehaviortimeinhomogeneous2021} and references therein. In this case, The velocity process can be studied independently on the position process. Even in the time-homogeneous case, various asymptotic behaviors can appear. Whenever the random force is supposed to be Brownian, a particular non-linear Langevin's type SDE was studied in \cite{FournierOnedimensionalcritical2021}: 
\[V_t=v_0+B_t-\dfrac{\rho}{2}\int_0^t\dfrac{V_s}{1+V_s^2}\dd s \quad \mbox{and}\quad X_t=x_0+\int_0^tV_s\dd s. \]
In that case, the frictional force asymptotically behaves as $-\dfrac{\rho}{v}$, which induces the velocity process to "behave", far away from zero, like a (signed) Bessel process of dimension $1-\rho$. Various asymptotic behaviors of the position process appear, depending on the moment order of Bessel excursion area (which depends itself on the value of $\rho$). More precisely, when $\rho\geq 5$, the moment is of order $2$, hence, using a suitable rescaling, the authors show that the position process behaves asymptotically as a Brownian motion. An $\alpha$-stable process appears as limiting dynamics when $\rho\in [1,5)$. The index of stability $\alpha$ is a function of $\rho$, which interpolates the power of the rescaling from $\frac{1}{2}$ (Brownian motion) to $\frac{3}{2}$ (integrated Bessel process). This last behavior occurs when $\rho \in(0,1)$. However, the tools used in \cite{FournierOnedimensionalcritical2021}, such as invariant measure, scale function and speed measure, are limited to time-homogeneous coefficients.

\medskip
In \cite{GradinaruExistenceasymptoticbehaviour2013}, \cite{GradinaruAsymptoticbehaviortimeinhomogeneous2021a} and \cite{GradinaruAsymptoticbehaviortimeinhomogeneous2021}, the drift coefficient $b$ is allowed to depend on time under an homogeneity condition. More precisely, the following system is considered
\begin{equation*}
	\begin{cases}
		\dd V_t=\dd L_t - \rho \dfrac{\sgn(V_t)\abs{V_t}^{\gamma}}{t^{\beta}}\dd t,\\
		\dd X_t= V_t \dd t.
	\end{cases}
\end{equation*} 
The frictional force is time-inhomogeneous, depending on non-negative parameters $\beta$, $\gamma$ and $\rho$. When the particle moves slowly, classical mechanics ensures that the frictional force is linear, i.e.\ $\gamma =1$. Whereas in the turbulent regime, when the particle moves faster, thanks to fluid dynamics, the frictional force depends quadratically on the velocity, i.e.\ $\gamma =2$. That is why in a broader framework, we assume that the frictional force has a space component of the form $v\mapsto -\rho\sgn(v)\abs{v}^{\gamma}$. Moreover, the frictional force can depend on time through a friction coefficient $t\mapsto \rho_t$. For a particle evolving in a fluid, it can be the case for example when the viscosity of the fluid or the geometry of the particle change with time. For this reason, a time dependence is added to the function $b$ in \cite{GradinaruExistenceasymptoticbehaviour2013}, \cite{GradinaruAsymptoticbehaviortimeinhomogeneous2021a} and \cite{GradinaruAsymptoticbehaviortimeinhomogeneous2021}. In these works, it is assumed that $\rho_t = \frac{\rho}{t^{\beta}}$. The main goal behind the study of this model is to understand the competition between the frictional force, which tends to immobilize the system and the random force perturbing it. Notice that, by the self-similarity property satisfied by $L$, $\mathbb{E}\left[ |L_t|\right]$ is proportional to $t^{\frac{1}{\alpha}}$. This shows that the noise $L_t$ acts with a typical scale $t^{\frac{1}{\alpha}}$ and thus, when $\alpha$ decreases, it perturbs the velocity with higher typical values. The interest of the works mentioned above is to study the long-time behavior of the system through the prism of the competition between these two opposite actions.\\ 

Let us mention two relevant examples in the Brownian case before explaining the results obtained in \cite{GradinaruAsymptoticbehaviortimeinhomogeneous2021a,GradinaruAsymptoticbehaviortimeinhomogeneous2021}. When $\beta=0$, the friction coefficient does not decrease with time. By ergodicity, the velocity converges towards its invariant distribution and thus, the rescaled position process $(\eps^{\frac12}X_{t/\eps})_t$ behaves as a Brownian motion as $\eps$ tends to $0$. When "$\beta=+ \infty$", i.e.\ when there is no frictional force, the rescaled velocity-position process $(\eps^{\frac12}V_{t/\eps},\eps^{\frac32}X_{t/\eps})_t$ converges in distribution towards $\left(B_t, \int_0^t B_s \dd s \right)_t.$ When $\beta >0$, the frictional force is evanescent: it slows down the system but less and less efficiently as time increases and we expect a transition between the two extreme cases mentioned above, both on the limiting processes and on the rescaling.\\

In \cite{GradinaruExistenceasymptoticbehaviour2013}, the authors study the convergence in distribution, when $t$ tends to $+\infty$, of $r_t V_t$, for a certain rate of convergence $r_t$ in the case where $L$ is a Brownian motion. In \cite{GradinaruAsymptoticbehaviortimeinhomogeneous2021a}, the authors extend the results obtained in \cite{GradinaruExistenceasymptoticbehaviour2013} to the whole process given by the couple velocity-position. Namely, the authors study the limit in distribution of the rescaled process $(r_{\eps,V} V_{t/\eps}, r_{\eps,X} X_{t/\eps})_t$ for two appropriate rates of convergence $r_{\eps,V}$ and $r_{\eps,X}$. Results were further generalized in \cite{GradinaruAsymptoticbehaviortimeinhomogeneous2021} to an $\alpha$-stable driving process. To be more precise, the authors highlight three regimes, depending on the balance between $\beta$, $\gamma$ and $\alpha$, the index of stability of $L$. \begin{itemize} \item Whenever the frictional force is sufficiently "small at infinity", i.e. if $\beta$ is large enough, the rescaled process behaves as if there was no frictional force. It thus converges in distribution towards the Kolmogorov process $(L, \int_0^\cdot L)$, as in the particular case "$\beta = + \infty$" mentioned above in the Brownian case and with the same rescaling $(r_{\eps,V},r_{\eps,X}) = (\eps^{\frac{1}{\alpha}}, \eps^{1 + \frac{1}{\alpha}})$. \item When the two forces offset, the rescaling remains the same as in the preceding regime and the limiting process is still of kinetic form $(\mathcal{V},\int_0^\cdot\mathcal{V})$, but the process $\mathcal{V}$ is henceforth ergodic. \item  Whereas, when the drag force swings with the random process, i.e.\ when $\beta$ is small enough, the limiting process is no longer kinetic and the rescaling is not the same as in the two preceding regimes. The rescaled velocity process converges in finite dimensional distributions towards a white noise. Here, the asymptotic behavior is somehow an interpolation between the two extreme cases $\beta =0$ and "$\beta = +\infty$", which is explained by the slow decrease of the frictional force with time. \end{itemize}
The proofs are essentially based on the self-similarity of the driving process and on moment estimates of the velocity process. \\

\bigskip
In this paper, we are interested in the long-time behavior of the solution to the following system of SDEs, defined on the time interval $[t_0,+\infty)$, where $t_0>0$ and $x_0,v_0 \in \mathbb{R}$ are fixed
\begin{equation}\label{eq: confined}\tag{SKE}
	\begin{cases}
		\dd V_t= \dd L_t -\sgn(V_t)\dfrac{\abs{V_t}^{\gamma}}{t^{\beta}}\dd t - X_t \dd t, \\
		\dd X_t= V_t\dd t,                                                                 \\ (V_{t_0},X_{t_0})=(v_{0},x_{0}),
	\end{cases}
\end{equation}
Here $\gamma,\beta>0$ and $L$ is a symmetric $\alpha$-stable process on $\mathbb{R}$ with $\alpha \in (1,2]$. More precisely, our goal is to study the asymptotic behavior, as $\eps \to 0$, of the rescaled velocity-position process \[(Z^{(\eps)}_t)_{t} := \left(r_{\eps}\begin{pmatrix}X_{t/\eps}\\ V_{t / \eps}\end{pmatrix}\right)_t,\] for an appropriate rate of convergence $r_{\eps}$. Our first motivation is to study how the presence of the quadratic potential influences the results obtained in \cite{GradinaruAsymptoticbehaviortimeinhomogeneous2021a,GradinaruAsymptoticbehaviortimeinhomogeneous2021} through a confining effect on the position $X$. Indeed, the confining effect is here related to the position of the particle and does not disappear asymptotically contrary to the frictional force. It has thus an effect both on the limiting processes and on the rescaling. Here, it is a competition of the quadratic potential and the frictional force, which confines and slows down the system, against the noise which perturbs it.\\

Notice that our system without noise and frictional force is nothing else than the classical harmonic oscillator \[\begin{cases}
	v'_t=-x_t,
	\\ x'_t=v_t.
\end{cases}\] The intrinsic oscillatory behavior induced by the quadratic potential prevents the rescaled process $Z^{(\eps)}$ from converging as a process. However, we prove that each of its one-dimensional marginal distributions converges. In order to obtain the convergence of the whole process, the key idea is to remove the oscillations present in the system. Namely, we set for $t \geq t_0$ \[ \Theta_{t} := \begin{pmatrix}
	\cos(t) & \sin(t) \\ -\sin(t) & \cos(t)
\end{pmatrix} \quad \text{and} \quad Y_t:= \Theta_{t}^{-1} \begin{pmatrix}
	X_t\\ V_t
\end{pmatrix},\] where $\Theta_{t}$ is the rotation on $\mathbb{R}^2$ of angle $-t$ and we study the behavior of $(r_{\eps} Y_{t/\eps})_t$ as $\eps$ tends to $0$, for a certain rate of convergence $r_{\eps}$.

\subsection{Notations, main results and comments} 
Let us first introduce some notations used throughout the paper. For simplicity, we shall write $\mathcal{C}$ and $\mathcal{D}$ respectively for $\mathcal{C}((0,+\infty), \mathbb{R}^2)$, the space of continuous functions defined on $(0,+\infty)$ and $\mathcal{D}((0,+\infty),\mathbb{R}^2)$, the Skorokhod space of functions defined on $(0,+\infty)$ which are càdlàg on every compact subinterval of $(0,+\infty)$. For $x,y \in \mathbb{R}^2$, $\absd{x}$ represents the Euclidean norm of $x$, and $x\cdot y$ the inner product of $x$ and $y$. If $x\in\mathbb{R}^2$, for each $i\in \{{1,2}\}$, $x^{(i)}$ denotes its $i$-th component. The minimum between two reals is denoted by $\wedge$. We call $I_2$ the identity matrix of dimension 2 and $A^T$ is the transpose matrix of a matrix $A$. Finally, we denote by $C$ some positive constant, which may change from line to line, and we use subscripts to indicate the parameters on which it depends when it is necessary. For the sake of simplicity, we denote by $C_{t_0}$ a positive constant depending only on $t_0,x_0$ and $v_0$, which are fixed throughout the paper. \\

We can now state our results. The following theorem deals with convergences in distribution in the space ${\mathcal{C}}$ endowed with topology of uniform convergence on every compact set of $(0,+\infty)$.
\begin{theo}[Brownian case i.e.\ $\alpha=2$]\label{ThmlongtimeY_mB} Define $q:= \frac{\beta}{\gamma+1}$, $r_{\eps}:= \eps^{q \wedge \frac12}$ and set $(Y_t^{(\eps)})_{t\geq\eps t_0}:=\left(r_{\eps}\Theta^{-1}_{t/\eps}(X_{t/\eps},V_{t/\eps})^T\right)_{t\geq \eps t_0}$. Let $\mathcal{B}$ be a standard two-dimensional Brownian motion on $\mathbb{R}^2$.
	\begin{enumerate}[label=(\roman*)]
		\item (Super-critical regime i.e.\ $2q>1$). The rescaled process $Y^{(\eps)}$ converges in distribution towards $\left(\mathcal{B}_{\frac{t}{2}} \right)_{ t > 0}$.
		\item (Critical regime i.e.\ $2q=1$). Assume that $\gamma=1$. The rescaled process $Y^{(\eps)}$ converges in distribution towards $\left(\frac{1}{\sqrt{2t}}\int_{0}^t\sqrt{s}\dd \mathcal{B}_{s} \right)_{ t >0}$, which is the centered Gaussian process with covariance kernel $K(s,t)= \frac{(s\wedge t)^{2}}{4\sqrt{st}}I_2$.
		\item (Sub-critical regime i.e.\ $2q<1$). Assume that $\gamma=1$ and $\beta \in \left(\frac{1}{2},1\right)$. The rescaled process $Y^{(\eps)}$ converges in finite dimensional distributions towards the centered Gaussian process with covariance kernel $K(s,t)= \frac{1}{2}s^{\beta}\mathbb{1}_{\{s=t\}}I_2$. 
	\end{enumerate}
\end{theo}
Let us denote by $\psi$ the characteristic exponent of the symmetric stable process $L$. It follows from  \citer[Theorem $14.15$ p. 86]{SatoLevyprocessesinfinitely1999} that there exists $a>0$ such that for all $\xi \in \mathbb{R}$,
\begin{equation}\label{eq:symbol_stable}\psi(\xi) = -a|\xi|^\alpha.\end{equation}

In the next theorem, the convergences occur in the space ${\mathcal{D}}$ endowed with the Skorokhod metric. 
\begin{theo}[Stable case i.e.\ $\alpha\in (1,2)$]\label{ThmlongtimeY_stable} Assume that $\gamma \in (0, \alpha)$. Define $q:= \frac{\beta}{\gamma+\alpha-1}$,  $r_{\eps}:= \eps^{q \wedge \frac{1}{\alpha}}$ and set $(Y_t^{(\eps)})_{t\geq\eps t_0}:=\left(r_{\eps}\Theta^{-1}_{t/\eps}(X_{t/\eps},V_{t/\eps})^T\right)_{t\geq \eps t_0}$. Let $\mathcal{L}$ be the rotationally invariant stable process on $\mathbb{R}^2$, whose characteristic exponent is given by 
	\[ \xi \in \mathbb{R}^2 \mapsto -\widetilde{C}\absd{\xi}^\alpha, \quad \mbox{with } \widetilde{C}:= \frac{a}{2\pi}\int_{0}^{2\pi} \abs{\cos(x)}^{\alpha}\dd x.\]
	\begin{enumerate}[label=(\roman*)]
		\item (Super-critical regime i.e.\ $\alpha q>1$). The rescaled process $Y^{(\eps)}$ converges in distribution towards $\left(\mathcal{L}_{t} \right)_{ t > 0}$.
		\item (Critical regime i.e.\ $\alpha q=1$). Assume that $\gamma=1$. The rescaled process $Y^{(\eps)}$ converges in distribution towards the Lévy-type process $\left(\frac{1}{\sqrt{t}}\int_{0}^t\sqrt{s}\dd \mathcal{L}_{s} \right)_{ t >0}$.
		\item (Sub-critical regime i.e.\ $\alpha q<1$). Assume that $\gamma=1$ and $\beta \in \left(\frac{1}{2},1\right)$. Then, for all $(t_1, \cdots, t_d)\in (0, +\infty)^{d}$, $\left(Y^{(\eps)}_{t_1}, \cdots, Y^{(\eps)}_{t_d}\right)$ converges in distribution towards the product measure $\mu_{t_1} \otimes \cdots \otimes \mu_{t_d}$, where $\mu_{t}$ is the distribution with characteristic function 
		\[\xi \in \mathbb{R}^2 \mapsto \exp\left(-\frac{2}{\alpha}\widetilde{C}\absd{\xi}^{\alpha}t^{\beta}\right).\]
	\end{enumerate}
\end{theo}
\begin{remark}
	
	The symmetry of $L$ is only required to ensure the well-posedness of \eqref{eq: confined} when $\gamma <1$ using \cite{MarinoWeakwellposednessdegenerate2021}.
 
\end{remark}

\begin{remark}
	At first sight, the parameter $\alpha$ does not seem to appear in the limiting process of Theorem \ref{ThmlongtimeY_stable} $(ii)$ contrary to the Brownian case where there is a factor $\frac{1}{\sqrt{2}}$ in front of the stochastic integral. The reason is that the constant $\alpha$ is hidden in the constant $\widetilde{C}$ which defines the $\alpha$-stable process $\mathcal{L}$ in Theorem \ref{ThmlongtimeY_stable}. Indeed, if we formally take $\alpha =2$ in Theorem \ref{ThmlongtimeY_stable} $(ii)$, we recover the limiting process of Theorem \ref{ThmlongtimeY_mB} $(ii)$. Let us justify it. When $\alpha =2$, the constant $\widetilde{C}$ can be computed explicitly and is equal to $\frac{1}{4}$. As the characteristic exponent of $\mathcal{L}$ is given by $$ \xi \in \mathbb{R}^2 \mapsto - \frac{\Vert \xi \Vert^2}{4},$$ we deduce that $(\mathcal{L}_t)_{t\geq 0} = (\mathcal{B}_{t/2})_{t \geq 0}$ in distribution. When $\alpha =2$, the limiting process in \ref{ThmlongtimeY_stable} $(ii)$ is thus given, for all $t > 0$, by $$ \frac{1}{\sqrt{2t}} \int_0^t \sqrt{s} \sqrt{2}\dd \mathcal{B}_{s/2},$$ which is equal in distribution to the limiting process of Theorem \ref{ThmlongtimeY_mB} $(ii)$ using the self-similarity of $\mathcal{B}$. We similarly notice that if we formally take $\alpha =2$ in Theorem \ref{ThmlongtimeY_stable} $(i)$ and $(iii)$, we recover the same limiting processes as in Theorem \ref{ThmlongtimeY_mB} $(i)$ and $(iii)$ with the same rescaling. 
	
\end{remark}
\begin{remark}
	The Hamiltonian process associated with the system is given by \[H_t:= \dfrac{1}{2}\abs{V_t}^2+\dfrac{1}{2}\abs{X_t}^2= \dfrac{1}{2}\absd{Z_t}^2=\dfrac{1}{2}\absd{Y_t}^2.\] Combining the preceding results with the continuous mapping theorem, we deduce the convergence of the rescaled energy process $(H^{(\eps)}_t)_{t>0}:=(r_{\eps}^2H_{t/\eps})_{t>0}$ as $\eps \rightarrow 0$ either as a process in the critical and super-critical regimes, or for finite dimensional distributions in the sub-critical regime. \\
	For example in the super-critical regime with $\alpha=2$, the limiting energy process $(H_t^{0})_{t\geq 0}:=(\frac{1}{2}\|\mathcal{B}_{\frac{t}{2}}\|^2)_{t\geq 0}$ is the squared Bessel process, which is the solution to the following equation
	\[\dd H_t^0=\sqrt{H_t^{0}}\dd B_t+\dfrac{1}{2}\dd t, \quad H_0^{0}=0,\]
	where $B$ is a standard one-dimensional Brownian motion.
	Note that we recover the limiting energy process obtained in \citer[Theorem 2.1]{AlbeverioLongtimebehavior1994} for the non-damped Hamiltonian system. The interpretation is that if the frictional force decreases sufficiently quickly as $t \rightarrow + \infty$, namely if $\beta$ is large enough, then the rescaled Hamiltonian process converges as if there were no damping. 
 
\end{remark}
We obtain furthermore the convergence in distribution as $t \to + \infty$ of $t^{-q\wedge \frac{1}{\alpha}} (X_t,V_t)^{T}$ in the following corollary.
\begin{coro}\label{Corollary_cv_Z}
	Let us define $(Z_t^{(\eps)})_{t\geq\eps t_0}:=(r_{\eps}(X_{t/\eps},V_{t/\eps})^T)_{t\geq \eps t_0}$, where $r_{\eps}:= \eps^{q \wedge \frac{1}{\alpha}}$. The rescaled process $Z^{(\eps)}$ does not converge in distribution. However, we deduce from Theorems \ref{ThmlongtimeY_mB} and \ref{ThmlongtimeY_stable} and under the same assumptions, the convergence in distribution of $r_{1/t}(X_t,V_t)^T$ towards explicit limits, as $t\to +\infty$.\\ 
	
	In the Brownian case, the limit is either $\mathcal{N}(0,\frac{1}{2} I_2)$ in the super-critical and sub-critical regimes, or $\mathcal{N}(0,\frac{1}{4} I_2)$ in the critical regime. \\
	
	In the stable case, keeping the same notations as in Theorem \ref{ThmlongtimeY_stable}, the characteristic function of the limit is given, for all $\xi \in \mathbb{R}^2$, by
	\begin{enumerate} [label=(\roman*)]
		
		\item $ \exp\left(-\widetilde{C}\absd{\xi}^{\alpha}\right)$ in the super-critical regime,
		\item $ \exp\left(-\left( 1 + \frac{\alpha}{2} \right)^{-1}\widetilde{C}\absd{\xi}^{\alpha}\right)$ in the critical regime,
		\item $ \exp\left(-\frac{2}{\alpha}\widetilde{C}\absd{\xi}^{\alpha}\right)$ in the sub-critical regime.
		
	\end{enumerate}
\end{coro}
The switch between the three regimes results in different scale parameters of the limiting distributions. Let us also notice that in the Brownian setting, the position $X$ and the velocity $V$ become independent in large time since the covariance matrix of the limiting Gaussian distribution is diagonal. However, this is false for the stable case. Indeed, the limit is a rotationally invariant stable distribution on $\mathbb{R}^2$, which cannot have independent coordinates.\\

As in \cite{GradinaruAsymptoticbehaviortimeinhomogeneous2021a,GradinaruAsymptoticbehaviortimeinhomogeneous2021}, we highlight three regimes for the asymptotic behavior of the system. However, the rate of convergence of the position $X$ is different from that found in \cite{GradinaruAsymptoticbehaviortimeinhomogeneous2021a,GradinaruAsymptoticbehaviortimeinhomogeneous2021}, when $\mathcal{U}=0$. Indeed, contrary to the free potential system, the position process is somehow more diffusive. This is due to the structure of our model. Namely, the presence of the quadratic potential allows the noise to propagate more efficiently from the velocity component to the position one (see \cite{FedrizziRegularitystochastickinetic2017} for more details). This explains why both the limiting processes and the rate of convergence are different between our work and \cite{GradinaruAsymptoticbehaviortimeinhomogeneous2021a, GradinaruAsymptoticbehaviortimeinhomogeneous2021}. 
Let us also note that the position process grows more slowly in our case than when $\mathcal{U}=0$. For example, in the Brownian super-critical regime, $X_t$ asymptotically behaves as $\mathcal{N}(0,\frac{t}{2})$ when $t$ tends to infinity in our framework, but as $\mathcal{N}(0,\frac{t^3}{3})$ in the free potential one. This difference can also be seen in moment estimates established for the position process $X$ (see Remarks \ref{rqmoment} and \ref{rqmoment1}). This is explained by the fact that the quadratic potential confines the position of the particle through a spring force.

\subsection{Strategy and plan of the paper}

In our model, the particle is no longer free, contrary to \cite{GradinaruAsymptoticbehaviortimeinhomogeneous2021a,GradinaruAsymptoticbehaviortimeinhomogeneous2021}, and the equations on the position and the velocity are intrinsically linked to each other.
Therefore, we can no longer separate by components the study of the velocity-position process. Writing the system \eqref{eq: confined} in a vector viewpoint, as done in \cite{FedrizziRegularitystochastickinetic2017}, we use a variation of constants method to return to the study of a two-dimensional system in a free potential. We then adapt the methods used in \cite{GradinaruAsymptoticbehaviortimeinhomogeneous2021a,GradinaruAsymptoticbehaviortimeinhomogeneous2021}. In the super-critical regime, the proof is essentially based on the self-similarity of the driving process and on moment estimates of $V$ and $X$. In the critical and sub-critical regimes, we need to restrict ourselves to a linear drag force, i.e.\ $\gamma =1$, in order to rely on the study of the asymptotic behavior of the solution to the underlying non-autonomous ordinary differential equation (ODE).
Whenever the driving process is Brownian, we take advantage of the theory of Gaussian processes. The convergence is thus characterized by the study of the mean and covariance functions. In the case of a stable driving process, we need to study the convergence, in distribution and as a process, of a Wiener-Lévy integral (i.e.\ the integral of a deterministic function integrated against a stable process). The key point here is to use the fact that a Wiener-Lévy integral is a process with independent increments.

\bigskip

Our paper is organized as follows. We consider the case of a Brownian driving process in \cref{s: mB}, and we follow the same structure for an $\alpha$-stable driving process in \cref{s: stable}. For the sake of clarity, we opt for separating the two cases since the tools used are different. Finally, we state and prove some technical results in \cref{appendix_ode} and \cref{appendix}.

\section{Study of the system driven by a Brownian motion}\label{s: mB}
In this section, the driving process $L$ is supposed to be a standard Brownian motion, i.e.\ $\alpha=2$. It will be denoted by $B$ to keep standard notations. To be precise, \eqref{eq: confined} becomes
\begin{equation}\label{eq: confined_mB}
	\begin{cases}
		\dd V_t= \dd B_t -\sgn(V_t)\dfrac{\abs{V_t}^{\gamma}}{t^{\beta}}\dd t - X_t \dd t, \\
		\dd X_t= V_t\dd t,                                                                 \\ (V_{t_0},X_{t_0})=(v_{0},x_{0}).
	\end{cases}
\end{equation}

The previous system can be written in a vector viewpoint. Indeed, we set, for all $t \geq t_0$ and $v \in \mathbb{R}$,
\[Z_t:=\begin{pmatrix}
	X_t \\ V_t
\end{pmatrix},\ W_t:=\begin{pmatrix}
	0 \\ B_t
\end{pmatrix},\ A:=\begin{pmatrix}
	0 & 1 \\ -1 & 0
\end{pmatrix},\ \Gamma:=\begin{pmatrix}
	0 & 0 \\ 0 & 1
\end{pmatrix}\ \mbox{and}\ F(t,v):=\begin{pmatrix} 0\\\sgn(v)\dfrac{\abs{v}^{\gamma}}{t^{\beta}}\end{pmatrix}.\]
Thereby, the system \eqref{eq: confined} can be rewritten as
\begin{equation}
	\begin{cases}\label{eq: Z_mB}
		\dd Z_t=\Gamma\dd W_t+ AZ_t\dd t - F(t,V_t)\dd t, \\ Z_{t_0} = z_0:=(x_{0},v_{0})^T.
	\end{cases}
\end{equation}
Notice that the matrix $A$ is the rotation matrix of angle $\frac{\pi}{2}$ and that, for all $t \in \mathbb{R}$, 
\[ \Theta_{t} :=e^{tA}=\begin{pmatrix}
	\cos(t) & \sin(t) \\ -\sin(t) & \cos(t)
\end{pmatrix}.\] 

We also define, for any $t \geq t_0$, $Y_t:=e^{-tA}Z_t$. We easily check, with Itô's formula, that $Y$ is given by
\begin{equation}\label{EDS_Y}\tag{$SDE_{Y}$}
	\dd Y_t= e^{-tA}\Gamma\dd W_t -e^{-tA}F(t,V_t)\dd t.
\end{equation}

\subsection{Existence up to explosion}

\begin{theo}\label{thm_existence_brownian}
	
	The system of SDEs \eqref{eq: confined_mB} admits a unique (global) strong solution if $\gamma \in (0,1]$. And if $\gamma > 1$, there exists a unique strong solution defined up to its explosion time $\tau_{\infty}$.
	
\end{theo}

\begin{proof}
	In the case $\gamma >1$, the coefficients of the SDE \eqref{eq: Z_mB} are locally Lipschitz continuous with respect to the space variable, uniformly in time. By a standard localization argument as mentioned in \cite[p. $417$, after $(13)$]{KallenbergFoundationsModernProbability2002}, there exists a unique solution up to explosion. We also refer to \cite[Theorem $2.8$ p. $154$]{MAO2011147} for a proof of the existence of a solution, up to explosion, in the more general case of path-dependent SDEs. \\
	
	Assume now that $\gamma \leq 1$. We will use \citer[Theorem 1]{HonoreStrongregularizationBrownian2018}. Keeping the same notations, we have in our case, for any $(x,v) \in \mathbb{R}^2$ and $t \geq t_0$, $F_1(t,v,x) := - \sgn(v) |v|^\gamma t^{-\beta} - x$, $ F_2(t,v,x) := v$ and $\sigma (t,v,x) =1$. Assumptions \textbf{(ML)} and  \textbf{(UE)} in \cite{HonoreStrongregularizationBrownian2018} are obviously satisfied. Let us now remark that $F_1$ is $\gamma$-Hölder with respect to $v \in \mathbb{R}$ uniformly with respect to $t\geq t_0$ and $x \in \mathbb{R}$, and is Lipschitz continuous with respect to $x$, uniformly with respect to $t$ and $v$. With the notations used in \cite{HonoreStrongregularizationBrownian2018}, we have $\beta_1 = \gamma$ and $\beta_2 = 1$. Thus, Assumption $\bm{(T_{\beta})}$ is satisfied. Finally, we check that Assumption $\bm{(H_{\eta})}$ is satisfied. Since $ \partial_{v} F_2 = 1$,  we can conclude, taking $\eta$ small enough and  $\mathcal{E}_1 = \{1\}$, that there exists a unique strong solution to \eqref{eq: confined_mB}.
	
\end{proof}

\subsection{Moment estimates and non-explosion}
In this section, we state and prove moment estimates of $Z$. It will be useful to control some stochastic terms appearing later.
For all $n\geq0$, define the stopping time
\begin{equation*}\label{stopping_time}
	\tau_n:=\inf \{ t\geq t_0, \ \absd{Z_t}\geq n\}.
\end{equation*}
Set $\tau_{\infty}:= \lim_{n\to +\infty}\tau_n$ the explosion time of $Z$.
\begin{prop}\label{prop: moment} The explosion time of $Z$ is a.s.\ infinite and, for all $\kappa \geq 0$ and $t \geq t_0$
	\begin{equation}\label{moment}
		\mathbb{E}\left[ \absd{Z_{t}}^{\kappa}\right] \leq  C_{\kappa, t_0}t^{\frac{\kappa}{2}}.
	\end{equation}
\end{prop}

\begin{remark}\label{rqmoment}
	Let us mention that the moment estimate obtained for the position process $X$ is a priori smaller in our case than in the free potential case \cite{GradinaruAsymptoticbehaviortimeinhomogeneous2021a}. It is explained by the confining effect of the quadratic potential.
\end{remark}
\begin{proof}The proof is adapted from \cite{GradinaruAsymptoticbehaviortimeinhomogeneous2021a} to two-dimensional processes. For the sake of completeness, we sketch the proof in our context. \\
	Using Itô's formula applied to the function $f:(x,v) \mapsto x^2 + v^2$ and the fact that for all $z \in \mathbb{R}^2$, $z \cdot Az = 0$, we deduce that, for all $t\geq t_0$,
	\[
	\absd{Z_{t\wedge \tau_n}}^2
	\leq
	\absd{z_0}^2+ \int_{t_0}^{t}2\mathbb{1}_{\{s\leq \tau_n\}} Z_s \cdot (\Gamma \dd W_s)-\int_{t_0}^{t\wedge \tau_n} 2Z_s\cdot F(s,V_s)\dd s
	+({t}-t_0).
	\]
	Remark that for any $s \geq t_0$, $Z_s\cdot F(s,V_s) = V_s\sgn(V_s)\abs{V_s}^{\gamma}s^{-\beta} \geq 0$.	Taking expectation yields
	\[ \mathbb{E}\left[ \absd{Z_{t\wedge \tau_n}}^2 \right]\leq \absd{z_0}^2+(t-t_0) \leq C_{t_0}t.\]
	Thanks to \cref{explosion}, we can conclude that the explosion time of $Z$ is a.s.\ infinite.
	Set $\kappa\in [0,2] $, so, by Jensen's inequality and Fatou's lemma
	\begin{equation}\label{estimates_confined}
		\mathbb{E}\left[ \absd{Z_{t\wedge \tau_{\infty}}}^{\kappa}\right]\leq \left(\mathbb{E}\left[ \absd{Z_{t\wedge \tau_{\infty}}}^{2}\right] \right)^{\frac{\kappa}{2}} \leq \left( \liminf_{n\to \infty}\mathbb{E}\left[ \absd{Z_{t\wedge \tau_n}}^2 \right] \right)^{\frac{\kappa}{2}} \leq  C_{\kappa, t_0}t^{\frac{\kappa}{2}}.
	\end{equation}
	This leads to \eqref{moment}.\\
	
	When $\kappa > 2$, $v\mapsto \absd{v}^{\kappa}$ is a $\mathcal{C}^2$-function, so by Itô's formula, for all $t\geq t_0$,
	\begin{multline*}
		\absd{Z_{t\wedge \tau_n}}^{\kappa} \leq \absd{z_0}^{\kappa}+ \int_{t_0}^{t\wedge \tau_n} \kappa \absd{Z_s}^{\kappa-2}Z_s\cdot (\Gamma\dd W_s) - \int_{t_0}^{t\wedge \tau_n} \kappa \absd{Z_s}^{\kappa-2} Z_s\cdot F(s,V_s)	\dd s \\+\int_{t_0}^{t\wedge \tau_n}C_{\kappa}\absd{Z_s}^{\kappa-2}\dd s.
	\end{multline*}
	In addition, it follows from the hypothesis on the sign of the drift function that
	\begin{equation}\label{Ito_confined}
		\absd{Z_{t\wedge \tau_n}}^{\kappa} \leq \absd{z_0}^{\kappa}+ \int_{t_0}^{t} \kappa \mathbb{1}_{\{s\leq \tau_n\}}\absd{Z_s}^{\kappa-2}Z_s\cdot (\Gamma\dd W_s) +\int_{t_0}^{t\wedge \tau_n}C_{\kappa}\absd{Z_s}^{\kappa-2}\dd s.
	\end{equation}
	Taking expectation in \eqref{Ito_confined}, we have
	\[	\mathbb{E}\left[\absd{Z_{t\wedge \tau_{\infty}}}^{\kappa} \right] \leq \liminf_{n\to \infty} \mathbb{E}\left[\absd{Z_{t\wedge \tau_n}}^{\kappa}   \right] \leq \absd{z_0}^{\kappa}+  \int_{t_0}^{t}C_{\kappa}\mathbb{E}\left[\absd{Z_s}^{\kappa-2}\right]\dd s.
	\]
	When $0\leq \kappa -2 \leq 2$, we can upper bound $\mathbb{E}\left[\absd{Z_s}^{\kappa-2}\right]$ by injecting \eqref{estimates_confined} and get
	\[\mathbb{E}\left[\absd{Z_{t\wedge \tau_{\infty}}}^{\kappa} \right]  \leq  \absd{z_0}^{\kappa}+  \int_{t_0}^{t}C_{\kappa, t_0}s^{\frac{\kappa-2}{2}}\dd s \leq C_{\kappa, t_0}s^{\frac{\kappa}{2}}.\]
	The method is then applied inductively to prove the inequality for all $\kappa>2$.
\end{proof}

\subsection{Asymptotic behavior of the solution}
We gather in this section the proof of \cref{ThmlongtimeY_mB}. The strategy is to prove the convergence of the finite dimensional distributions (f.d.d.) of the process $Y^{(\eps)}$, as $\eps \to 0$, and its tightness in the critical and super-critical regimes.
We first focus on the tightness.
\begin{lem}\label{lem: tight}
	If $2q\geq 1$, then the family $\left\{(\sqrt{\eps}Y_{t/\eps})_{t\geq \eps t_0 }, \ \eps >0\right\}$ is tight on every compact interval $[m,M]$, with $0<m\leq M$.
\end{lem}
\begin{proof}
	We use the Kolmogorov criterion stated in \citer[Problem 4.11 p. 64] {KaratzasBrownianMotionStochastic1998}.\\
	Take $\eps_0$ small enough such that for all $\eps\leq \eps_0$, we have $\eps t_0 \leq m$. Fix $m\leq s\leq t \leq M$ and $a>4$.
	Define, for $t\geq \eps t_0$, the local martingale term appearing in \eqref{EDS_Y}
	\begin{equation}\label{eq: local_martingale_mb}
		M_t^{(\eps)}:=\sqrt{\eps}\int_{t_0}^{t/\eps}e^{-sA}\Gamma\dd W_s= \sqrt{\eps}\int_{t_0}^{t/\eps}\begin{pmatrix}
			- \sin(s) \\
			\cos(s)
		\end{pmatrix}\dd B_s.
	\end{equation}
	
	Using Jensen's inequality, moment estimates (see \cref{prop: moment}) and Burkholder-Davis-Gundy's inequality (see \citer[Theorem 4.4.22 p. 263]{ApplebaumLevyprocessesstochastic2009}), we have
	\[
	\begin{aligned}
		\mathbb{E}\left[\absd{Y_t^{(\eps)}-Y_s^{(\eps)}}^{a}  \right] & \leq C_{a}\mathbb{E}\left[\absd{M_t^{(\eps)}-M_s^{(\eps)} }^{a} \right] +C_{a}\mathbb{E}\left[\absd{\sqrt{\eps}\int_{s/\eps}^{t/\eps}e^{-uA}F(u,V_u)\dd u }^{a} \right]
		\\ & \leq C_{a} \mathbb{E}\left[\absd{M_t^{(\eps)}-M_s^{(\eps)} }^{a} \right] +C_{a} \eps^{1-\frac{a}{2}}(t-s)^{a-1}\mathbb{E}\left[\int_{s/\eps}^{t/\eps}\absd{F(u,V_u)}^{a}\dd u  \right] \\ & \leq C_{a} \mathbb{E}\left[\left(\mathrm{Tr}\left(\left\langle M^{(\eps)}_{ \cdot}-M_{s }^{(\eps)} \right\rangle_t \right)\right)^{a/2} \right] +C_{a}\eps^{1-\frac{a}{2}}(t-s)^{a-1}\int_{s/\eps}^{t/\eps}u^{\frac{\gamma a}{2}-\beta a}\dd u   \\ & \leq C_{a} (t-s)^{\frac{a}{2}}+C_{a,m,M}\eps^{a(\beta-\frac{\gamma+1}{2})}(t-s)^{a-1}
		\\ & \leq C_{a,m,M}(t-s)^{\frac{a}{2}}.
	\end{aligned}
	\]
	Since $\frac{a}{2}>2$ and $\beta\geq \frac{\gamma+1}{2}$ the upper bound is independent of $\eps \leq 1$.
	Furthermore, by moment estimates (\cref{prop: moment}),
	\[\sup_{\eps\leq \eps_0}\mathbb{E}\left[\absd{Y_m^{(\eps)}} \right] \leq \sqrt{m }<\infty. \]
	Thus, Kolmogorov's criterion can be applied, proving the tightness result.
\end{proof}
We will now prove the convergence of the finite-dimensional distributions of $Y^{(\eps)}$. Thanks to the previous lemma, this will yield the weak convergence on every compact set (see \citer[Theorem 13.1 p. 139]{BillingsleyConvergenceprobabilitymeasures1999}). The convergence in distribution on the whole space $\mathcal{C}$ will follow, for $2q\geq 1$, from \citer[Theorem 16.7 p. 174]{BillingsleyConvergenceprobabilitymeasures1999}, since all processes considered are continuous.
\subsubsection{Convergence of the f.d.d.\ in the super-critical regime}
Assume here that $2q>1$. Recall that $(Y_t^{(\eps)})_{t\geq\eps t_0}:=(\sqrt{\eps}Y_{t/\eps})_{t\geq \eps t_0}$.
\begin{proof}[Proof of \cref{ThmlongtimeY_mB} \textit{(i)}] ~
	\begin{steps}
		\item  We first prove the convergence of the f.d.d.\ of the local martingale term $M^{(\eps)} =: (M^{(\eps,1)}, M^{(\eps,2)})^T$ appearing in \eqref{EDS_Y}.\\
		Recall that the stochastic integral $M^{(\eps)}$ was defined in \eqref{eq: local_martingale_mb}. It is a centered Gaussian process with covariance kernel defined, for any $(s,t)\in [\eps t_0,+\infty)^2$,  by
		\begin{equation*}
			K^{(\eps)}(s,t):=\begin{pmatrix}
				\Cov(M^{(\eps)}_s)               & \Cov(M^{(\eps)}_s,M^{(\eps)}_t ) \\
				\Cov(M^{(\eps)}_t,M^{(\eps)}_s ) & \Cov(M^{(\eps)}_t)
			\end{pmatrix},
		\end{equation*}
		where
		\begin{equation*}
			\Cov(M^{(\eps)}_s,M^{(\eps)}_t )= \begin{pmatrix}
				\Cov(M^{(\eps,1)}_s,M^{(\eps,1)}_t ) & \Cov(M^{(\eps,1)}_s,M^{(\eps,2)}_t) \\ \Cov(M^{(\eps,2)}_s,M^{(\eps,1)}_t) & \Cov(M^{(\eps,2)}_s,M^{(\eps,2)}_t )
			\end{pmatrix},
		\end{equation*}
		and $\Cov(M^{(\eps)}_s)= \Cov(M^{(\eps)}_s,M^{(\eps)}_s )$.
		Thus, the convergence of the f.d.d.\ of $M^{(\eps)}$ is reduced to the study of the limit of $K^{(\eps)}$, when $\eps$
		converges to $0$. Let us fix $\eps t_0\leq s\leq t$. Using that $M^{(\eps)}$ has independent increments and by Itô's isometry, we find that
		\[ \Cov(M^{(\eps)}_s,M^{(\eps)}_t ) = \begin{pmatrix} \eps \int_{t_0}^{s/\eps} \sin(u)^2 \dd u & -\eps \int_{t_0}^{s/\eps}  \sin(u)\cos(u) \dd u \\ -\eps \int_{t_0}^{s/\eps}  \sin(u)\cos(u) \dd u &  \eps \int_{t_0}^{s/\eps}  \cos(u)^2 \dd u \end{pmatrix}. \]
		We get that, for all $ 0 < s \leq t$, 
		\[ \Cov(M^{(\eps)}_s,M^{(\eps)}_t )\quad \underset{\eps \to 0}{\longrightarrow}\quad \dfrac{1}{2}\begin{pmatrix} s & 0\\0 &  s \end{pmatrix}. \] 
		We recognize the covariance kernel of the process $\left(\mathcal{B}_{\frac{t}{2}}\right)_{t>0}$, where $\mathcal{B}$ denotes a standard Brownian motion on $\mathbb{R}^2$. Since mean and covariance functions characterize Gaussian process (see \citer[Lemma 13.1 \textit{(i)} p. 250]{KallenbergFoundationsModernProbability2002}), we have thus proved that $(M^{(\eps)}_t)_{t\geq \eps t_0}$ converges in f.d.d.\ towards $\left(\mathcal{B}_{\frac{t}{2}}\right)_{t >0}$.
		\item Pick $T>0$. We prove that
		\begin{equation*}
			\mathbb{E}\left[\sup_{\eps t_0\leq t \leq T}\absd{Y_t^{(\eps)}- M_t^{(\eps)}}\right]\quad  \underset{\eps \to 0}{\longrightarrow}\quad 0.
		\end{equation*}
		Let us fix $\eps>0$ small enough such that $\eps t_0 \leq T$. We have
		\[\sup_{\eps t_0\leq t \leq T}\absd{Y_t^{(\eps)}-M_t^{(\eps)}} \leq \sqrt{\eps}\absd{z_0} +\sqrt{\eps}\int_{t_0}^{T/\eps}\absd{e^{-sA}F(s,V_s)}\dd s.  \]
		We use moment estimates (\cref{prop: moment}) to get
		\begin{align}\label{major_espectation}
			\mathbb{E}\left[\sqrt{\eps}\int_{t_0}^{T/\eps}\absd{e^{-sA}F(s,V_s)}\dd s\right] & \notag =\mathbb{E}\left[\sqrt{\eps}\int_{t_0}^{T/\eps}\absd{F(s,V_s)}\dd s\right] \\ & \notag\leq \mathbb{E}\left[\sqrt{\eps}\int_{t_0}^{T/\eps}\abs{V_s}^{\gamma}s^{-\beta}\dd s\right] \\ & \notag\leq \sqrt{\eps}C_{\gamma, t_0}\int_{t_0}^{T/\eps}s^{\frac{\gamma}{2}-\beta}\dd s \\ & \leq C_{\gamma, t_0} (\eps^{\beta-\frac{\gamma+1}{2}}T^{\frac{\gamma}{2}-\beta+1}+\sqrt{\eps}t_0^{\frac{\gamma}{2}-\beta+1}).
		\end{align}
		Hence, setting $r:= \min(\beta-\frac{\gamma+1}{2}, \frac{1}{2} )$, which is positive by assumption, we get
		\begin{equation*}
			\mathbb{E}\left[\sup_{\eps t_0\leq t \leq T}\absd{Y_t^{(\eps)}- M_t^{(\eps)}}\right] = \underset{\eps \to 0}{O}(\eps^r).
		\end{equation*}
		We conclude the proof using \citer[Theorem 3.1 p. 27]{BillingsleyConvergenceprobabilitymeasures1999}.
	\end{steps}
\end{proof}

\subsubsection{Convergence of the f.d.d.\ in the critical and sub-critical regimes}
In this section, we consider the linear case, i.e.\ $\gamma=1$. Pick $\beta\in \left(\frac{1}{2},1\right]$. Recall that 
\[(Y_t^{(\eps)})_{t\geq\eps t_0}:=(\eps^{q}Y_{t/\eps})_{t\geq \eps t_0},\]
where $q = \frac{\beta}{\gamma +1}$.

\begin{proof}[Proof of \cref{ThmlongtimeY_mB} $(ii)$ and $(iii)$]
	Leaving out the Brownian term, the underlying ODE of our system is the following
	\begin{equation}\label{eq: ode_sscritiq}
		x''(t)+\dfrac{x'(t)}{t^{\beta}}+x(t)=0, \quad t \geq t_0.
	\end{equation}
	
	Pick the basis of solutions given in \cref{eq: resolvante } and denote by $R$ its resolvent matrix which is the solution to
	\begin{equation*}
		R_t'= \begin{pmatrix}
			0 & 1 \\ -1& -\frac{1}{t^{\beta}}
		\end{pmatrix}R_t, \quad t \geq t_0.
	\end{equation*}
	It follows by differentiating the inverse function that \begin{equation*}
		(R^{-1})_t'= - R_t^{-1} \begin{pmatrix}
			0 & 1 \\ -1& -\frac{1}{t^{\beta}}
		\end{pmatrix}, \quad t \geq t_0.
	\end{equation*}
	Using Itô's product rule for $R_t^{-1}Z_t$, the fact that the quadratic covariation of  $(R_t^{-1})_{t \geq t_0}$ and $(Z_t)_{t\geq t_0}$ is equal to zero and \eqref{eq: Z_mB} with $\gamma =1$, we get that for all $t\geq t_0$,
	\begin{align*}
		R_t^{-1}Z_t & = R_{t_0}^{-1}Z_{t_0} + \int_{t_0}^t R_s^{-1} \dd Z_s + \int_{t_0}^t (R^{-1})_s' Z_s \dd s  \\ &= R_{t_0}^{-1}Z_{t_0}
		+\int_{t_0}^tR_s^{-1}\Gamma \dd W_s.
	\end{align*}Let us define $f$ the rate of decrease of $R$ (see \cref{eq: resolvante }) by
	\begin{equation}\label{defrateofdecrease_recall}
		\forall t >0, \, f(t):= \begin{cases}
			\frac{1}{\sqrt{ t}}                               & \mbox{if }\beta=1, \\
			\exp \left(-\frac{t^{1-\beta}}{2(1-\beta)}\right) & \mbox{else.}
		\end{cases}
	\end{equation}
	Set, for $t \geq \eps t_0$,
	\begin{equation}\label{eq: widetildeM}
		\Phi_t:= \dfrac{e^{-tA}R_t}{f(t)}\quad \mbox{and}\quad \widetilde{M}^{(\eps)}_t:= \eps^q f\left(\frac{t}{\eps}\right)\int_{t_0}^{t/\eps}R_s^{-1}\Gamma \dd W_s.
	\end{equation}
	Pick $t\geq \eps t_0$.
	To study the convergence of $Y^{(\eps)}$ we decompose it into
	\begin{equation}\label{reecritureYbm}
		Y_{t}^{(\eps)}= \eps^q f\left(\frac{t}{\eps}\right)\Phi_{t/\eps}R_{t_0}^{-1}Z_0+\Phi_{t/\eps}\widetilde{M}^{(\eps)}_t.
	\end{equation}

	\begin{steps}
		\item Convergence of $\Phi$ and simplification the problem.\\
		
		Using the asymptotic expansion of the resolvent matrix (\cref{eq: resolvante }), we can write, for $t\geq \eps t_0$,
		\begin{equation*}
			\Phi_t=  I_2 +\underset{t\to \infty}{O}\left(t^{1-2\beta}\right).
		\end{equation*}
		As a consequence, since $1-2\beta<0$, $\Phi_{t/\eps}$ converges to the identity matrix $I_2$, as $\eps \to 0$ and for any $t>0$. Let us notice that, for any $t>0$, $\eps^q f\left(\frac{t}{\eps}\right)$ converges to 0, as $\eps \to 0$. We thus obtain that 
		\[\eps^q f\left(\frac{t}{\eps}\right)\Phi_{t/\eps}R_{t_0}^{-1}Z_0 \quad  \underset{\eps \to 0}{\longrightarrow} \quad 0.\]
		Therefore, we can forget the first term appearing in the decomposition \eqref{reecritureYbm} of $Y^{(\eps)}$ (see \citer[Theorem 3.1 p. 27]{BillingsleyConvergenceprobabilitymeasures1999}). It is thus enough to prove the convergence of the f.d.d.\ of the centered Gaussian process $(\Phi_{t/\eps}\widetilde{M}^{(\eps)}_t)_t$ to deduce the convergence of the f.d.d.\ of $Y^{(\eps)}$ towards the same limit. 
		
		\item Computation of the covariance kernel of $(\Phi_{t/\eps}\widetilde{M}^{(\eps)}_t)_t$ and convergence in f.d.d.\ \\

	We have, for all $\eps t_0 \leq s \leq t < + \infty$,
		\begin{equation}\label{expres_convariance_mb}\begin{aligned}
				\Cov\left(\Phi_{s/\eps} \widetilde{M}_s^{(\eps)}, \Phi_{t/\eps} \widetilde{M}_t^{(\eps)}\right)= \Phi_{s/\eps} \Cov(\widetilde{M}_{s}^{(\eps)},\widetilde{M}_{t}^{(\eps)})  \Phi_{t/\eps}^T
			\end{aligned}. \end{equation}
		Using the expression of the Wronskian obtained in \cref{eq: resolvante }, we obtain, for all $t\geq \eps t_0$,
		\begin{equation*}
			\widetilde{M}^{(\eps)}_t = \eps^q f(t/\eps)\int_{t_0}^{t/\eps}f(u)^{-2}\begin{pmatrix}
				-y_2(u) \\ y_1(u)
			\end{pmatrix} \dd B_u.
		\end{equation*}
		It is a centered Gaussian process and for any $\eps t_0 \leq s \leq t$, we have
		\begin{equation*}
			\Cov(\widetilde{M}^{(\eps)}_s,\widetilde{M}^{(\eps)}_t )= \eps^{\beta} f(t/\eps)f(s/\eps)\int_{t_0}^{s/\eps} f(u)^{-4}\begin{pmatrix}
				y_2^2(u) & -y_2(u)y_1(u) \\ -y_2(u)y_1(u) & y_1^2(u)
			\end{pmatrix}\dd u.
		\end{equation*}
		Using the asymptotic expansion of the solutions and \cref{lemmaintegralasymptoticexpansion}, we get, for all $\eps t_0 < s \leq t$,
	
		\begin{multline*}
			\Cov(\widetilde{M}^{(\eps)}_s,\widetilde{M}^{(\eps)}_t )=\eps^{\beta} f(t/\eps)f(s/\eps)\int_{t_0}^{s/\eps} f(u)^{-2}\begin{pmatrix}
				\sin^2(u) & -\sin(u)\cos(u) \\ -\sin(u)\cos(u) & \cos^2(u)
			\end{pmatrix}\dd u \\+ \underset{\eps \to 0}{O}\left(\eps^{2\beta-1}f(t/\eps) f(s/\eps)^{-1}  \right).
		\end{multline*}
		Moreover, using asymptotic expansions of these integrals (see Lemmas \ref{lemma_average_periodic2} and \ref{lemmaintegralasymptoticexpansion}),
		\begin{multline*}
			\eps^{\beta} f(t/\eps)f(s/\eps)\int_{t_0}^{s/\eps}f(u)^{-2}\cos^2(u)\dd u =  \eps^{\beta} f(t/\eps)f(s/\eps)\dfrac{1}{2}\int_{t_0}^{s/\eps}f(u)^{-2}\dd u \\+\underset{\eps \to 0}{o}\left( f(t/\eps)f(s/\eps)^{-1}\right).
		\end{multline*}
		The same equality holds for
		\[\eps^{\beta} f(t/\eps)f(s/\eps)\int_{t_0}^{s/\eps}f(u)^{-2}\sin^2(u)\dd u, \]
		and we have
	
		\begin{equation*}
			\eps^{\beta} f(t/\eps)f(s/\eps)\int_{t_0}^{s/\eps}f(u)^{-2}\cos(u)\sin(u)\dd u =\underset{\eps \to 0}{o}\left(f(t/\eps)f(s/\eps)^{-1} \right).
		\end{equation*}
		Thanks to \cref{lemmaintegralasymptoticexpansion}, this leads to
		\begin{equation*}
			\begin{aligned}
				\Cov(\widetilde{M}^{(\eps)}_s,\widetilde{M}^{(\eps)}_t ) & = \left[\dfrac{1}{2} \eps^{\beta} f(t/\eps)f(s/\eps)\int_{t_0}^{s/\eps}f(u)^{-2}\dd u \right]I_2 + \underset{\eps \to 0}{o}\left(f(t/\eps)f(s/\eps)^{-1} \right) \\ & = k_{\beta}
				\frac{f(t/\eps)}{f(s/\eps)}s^{\beta}I_2
				+\underset{\eps \to 0}{o}\left( f(t/\eps)f(s/\eps)^{-1} \right),
			\end{aligned}
		\end{equation*}
		where
		\begin{equation*}
			k_{\beta}:=\begin{cases}
				\frac{1}{4} & \mbox{ if }\beta =1, \\
				\frac{1}{2} & \mbox{ else. }
			\end{cases}
		\end{equation*}
		
		It follows from the definition of $f$ given in \eqref{defrateofdecrease_recall} and the fact that $s \leq t$ that  \begin{equation*}
			\Cov(\widetilde{M}^{(\eps)}_s,\widetilde{M}^{(\eps)}_t ) \underset{\eps \to 0}{\longrightarrow} \begin{cases}
				\frac{1}{4}\frac{(s\wedge t)^2}{\sqrt{st}} I_2 & \mbox{ if }\beta =1, \\
				\frac{1}{2} s^\beta \mathbb{1}_{\{s=t\}} I_2& \mbox{ else. }
			\end{cases}
		\end{equation*}
		
		Using the preceding convergence, \eqref{expres_convariance_mb} and Step $1$, we have proved the convergence of the f.d.d.\ of $Y^{(\eps)}$. Note that whenever $2q = 1$, i.e.\ $\beta =1$ since $\gamma =1$, we recognize the covariance kernel of the process $\left(\frac{1}{\sqrt{2t}}\int_{0}^t\sqrt{s}\dd \mathcal{B}_{s}\right)_{t>0}$, where $\mathcal{B}$ denotes a standard Brownian motion on $\mathbb{R}^2$.
		
	\end{steps}
\end{proof}
\begin{remark}
	\begin{itemize}
		\item The proof relies on the asymptotic expansion of the resolvent matrix of \eqref{eq: ode_sscritiq}. We were able to prove it only for $\beta \in \left(\frac{1}{2},1\right]$. However, if $\beta =0$, the resolvent matrix is explicit and following the same lines, we can prove that $\left(Z_{t/\eps}\right)_{t\geq \eps t_0}$ converges in f.d.d.\ towards a centered Gaussian process with covariance kernel $(s,t)\mapsto \frac{1}{2}I_2\mathbb{1}_{\{s=t\}}$. This behavior can be explained by the fact that the frictional force does not decrease along time. This cancels somehow the rotation bearing, which prevents $Z^{(\eps)}$ from converging as a process when $\beta >0$.
		\item The asymptotic expansion of the resolvant matrix is also known in the super-critical regime, i.e. $\beta>1$. Therefore, one can prove the result in the linear case, i.e. $\gamma=1$, following the same lines. 
	\end{itemize}
	
\end{remark}
\subsection{Proof of \cref{Corollary_cv_Z}}

\begin{proof}[Proof of \cref{Corollary_cv_Z}]
	We start by proving the convergence in distribution of $r_{1/T}Z_T$, as $T\to +\infty$. We claim that it follows from \cref{ThmlongtimeY_mB}.	Indeed, it is enough to remark that the convergence results stated in \cref{ThmlongtimeY_mB} imply the convergence in distribution of the marginal distribution at time $t=1$ of $Y^{(\eps)}$. Let us also recall that $Z_T = e^{TA} Y_T$. Setting $T=\frac{1}{\eps}$, the convergence of $r_{1/T}Z_T$ is therefore a direct consequence of \cref{Lemma_oscillating_gaussian_convergence}.
	\\
	We now show that the rescaled process $Z^{(\eps)}$ does not converge in distribution. We do the proof only in the super-critical regime. Assume by contradiction that it is the case. Hence, each of its coordinates shall converge too. We thus have the convergence of the rescaled process $X^{(\eps)}$.
	Using \eqref{EDS_Y}, we can write
	\begin{equation*}
		\sqrt{\eps}X_{t/\eps}= \sqrt{\eps}x_0+\sqrt{\eps}\int_{t_0}^{t/\eps} \sin\left(\frac{t}{\eps}-s \right)\dd B_s-\sqrt{\eps}\int_{t_0}^{t/\eps}\sin\left(\frac{t}{\eps}-s \right)F(s,V_s)\dd s.
	\end{equation*}
	As in the proof of \cref{ThmlongtimeY_mB} (i), the last term converges in probability uniformly on compact intervals towards zero. Hence, the following term shall converge in distribution
	\begin{equation*}
		I_t^{(\eps)}:=\sqrt{\eps}\int_{t_0}^{t/\eps} \sin\left(\frac{t}{\eps}-s \right)\dd B_s.
	\end{equation*}
	The process $(I_t^{(\eps)})_{t\geq \eps t_0}$ is Gaussian, thereby its limit shall be Gaussian too and its covariance function shall converge (see \citer[Lemma 13.1 \textit{(i)} p. 250]{KallenbergFoundationsModernProbability2002}).
	However, using Itô's isometry, one can compute, for $\eps t_0\leq s \leq t$,
	\[\begin{aligned}
		\mathbb{E}\left[I_t^{(\eps)}I_s^{(\eps)} \right] & =  \eps \int_{t_0}^{s/\eps} \sin\left(\frac{t}{\eps}-u \right)\sin\left(\frac{s}{\eps}-u \right)\dd u \\ & = \eps  \frac{1}{2}\left[\cos\left(\frac{t-s}{\eps}\right)\left(\frac{s}{\eps}-t_0\right)+\frac{1}{2}\left(\sin\left(\frac{t-s}{\eps}\right)-\sin\left(\frac{t+s}{\eps} -2t_0\right)\right) \right] \\ & = \frac{1}{2}s\cos\left(\frac{t-s}{\eps}\right) + \underset{\eps \to 0}{o}(1).
	\end{aligned} \]
	This term does not converge if $s\neq t$, and that concludes the proof.
\end{proof}

\section{Study of the system driven by an $\alpha$-stable process}\label{s: stable}
In this section, $L$ is a symmetric $\alpha$-stable Lévy process. We call $\nu$ its Lévy measure, which can be written as $\nu(\dd z)= a\abs{z}^{-1-\alpha}\mathbb{1}_{\{z\neq 0\}}\dd z$ with $a >0$. As a Lévy measure, it satisfies $\int_{\mathbb{R}^*}(1\wedge z^2)\nu(\dd z)<+\infty$. 
We denote by $N$ the Poisson random measure associated with $L$ and by $\widetilde{N}$ its compensated Poisson measure. Using Lévy-Itô's decomposition (see \cite[Remark $14.6$ and Theorem $14.7$ iii)]{SatoLevyprocessesinfinitely1999}), we have, for all $t\geq0$,
\begin{equation}\label{eq:decomposition_stable}
	L_t= 	   \int_0^t\int_{\mathbb{R}^*}z\widetilde{N}(\dd s, \dd z).                                                     
\end{equation}
As in the previous section, we set, for all $t \geq t_0$ and $v \in \mathbb{R}$,
\[Z_t:=\begin{pmatrix}
	X_t \\ V_t
\end{pmatrix},\ S_t:=\begin{pmatrix}
	0 \\ L_t
\end{pmatrix},\ A:=\begin{pmatrix}
	0 & 1 \\ -1 & 0
\end{pmatrix},\ \Gamma:=\begin{pmatrix}
	0 & 0 \\ 0 & 1
\end{pmatrix}\ \mbox{and}\ F(t,v):=\begin{pmatrix} 0\\\sgn(v)\dfrac{\abs{v}^{\gamma}}{t^{\beta}}\end{pmatrix}.\]
Thereby, the system \eqref{eq: confined} can be rewritten as
\begin{equation}
	\begin{cases}\label{eq: Z}
		\dd Z_t=\Gamma\dd S_t+ AZ_t\dd t - F(t,V_t)\dd t, \\ Z_{t_0} = z_0:=(x_{0},v_{0})^T.
	\end{cases}
\end{equation}
We define, for any $t \geq t_0$, $Y_t:=e^{-tA}Z_t$. We easily check, with Itô's formula, that $Y$ is given by
\begin{equation}\label{EDS_Y_Levy}\tag{$SDE_{Y}$}
	\dd Y_t= e^{-tA}\Gamma\dd S_t -e^{-tA}F(t,V_t)\dd t.
\end{equation}
\subsection{Existence up to explosion}

\begin{theo}\label{thm_existence_stable}
	The system \eqref{eq: confined}
	admits a unique weak solution if $\gamma \in (0,1]$. If $\gamma > 1$, there exists a unique strong solution defined up to its explosion time $\tau_{\infty}$.
	
\end{theo}

\begin{proof}
	In the case $\gamma >1$, the coefficients of the SDE \eqref{eq: Z} satisfied by $Z=(X,V)$ are locally Lipschitz continuous with respect to the space variable, uniformly in time. So we can apply \citer[Theorem $6.2.11$]{ApplebaumLevyprocessesstochastic2009} (see \citer[Theorem $119$]{SituTheoryStochasticDifferential2005} for a detailed proof), which ensures the existence of a unique solution to \eqref{eq: Z} defined up to explosion.\\
	
	Assume now that $\gamma \leq 1$. We check that we can use \citer[Theorem 1]{MarinoWeakwellposednessdegenerate2021}. Using the same notations, we have 
	\[\widetilde{A} = \begin{pmatrix}
		0 & -1 \\ 1 & 0
	\end{pmatrix},\]
	and for any $(t,x_1,x_2) \in [t_0, + \infty) \times \mathbb{R}^2$, $F_1(t,x_1,x_2)= -\sgn(x_1)|x_1|^\gamma t^{-\beta}$, $ F_2(t,x_1,x_2)=0$ and $\sigma(t,x_1,x_2) =1$. Assumptions \textbf{(UE)} and \textbf{(ND)} are clearly satisfied. Since $F_2$ does not depend on $x_1$ and since $[\widetilde{A}]_{2,1}=1$ is different from $0$, we deduce that Assumption \textbf{(H)} is satisfied. We easily check that \citer[Theorem 1]{MarinoWeakwellposednessdegenerate2021} can be applied with $\beta_1 = \gamma$, and $\beta_2=1$.
	
\end{proof}
\begin{remark} For $\alpha\in (0,2)$, employing the technique of Picard iteration and the interlacing procedure, one can deduce that \eqref{eq: Z} has a unique solution in the linear setting $\gamma=1$ (see \cite[p. 375]{ApplebaumLevyprocessesstochastic2009}).
\end{remark}

\subsection{Moment estimates and non-explosion}
Let $Z$ be the unique solution up to explosion time to \eqref{eq: Z}.
As in the continuous setting, define, for all $r\geq0$, the stopping time
\begin{equation*}\label{stopping_time_levy_confined}
	\tau_r:=\inf \{ t\geq t_0, \ \absd{Z_t}\geq r\}.
\end{equation*}
Set $\tau_{\infty}:= \lim_{r\to +\infty}\tau_r$ the explosion time of $Z$.
For the sake of simplicity, since there is no jump on the position component, for $z \in \mathbb{R}$, we shall write $Z_{s-}+z$ for $(X_{s},V_{s-}+z)$ in the following.\\
We adapt the proof of \cite{GradinaruAsymptoticbehaviortimeinhomogeneous2021} to two-dimensional processes.

\begin{prop}\label{esperance2} For any $\gamma\geq 0$ and $\beta\geq 0$, the explosion time $\tau_{\infty}$ is a.s.\ infinite and for $\kappa\in(0,\alpha)$, there exists $C_{\kappa,t_0}$ such that
	\begin{equation}\label{eq: moment 1_alpha}
		\ \forall t\geq t_0,\ \mathbb{E}\left[\absd{Z_{t}}^{\kappa}\right]\leq C_{\kappa,t_0}t^{\frac{\kappa}{\alpha}}.
	\end{equation}
\end{prop}
\begin{remark}\label{rqmoment1}
	Note that, as in the Brownian case, the moment estimates obtained for the position process $X$ is a priori smaller in our case than in the free potential case \cite{GradinaruAsymptoticbehaviortimeinhomogeneous2021}. It is explained by the confining effect of the quadratic potential.
\end{remark}
\begin{proof}
	The key idea is to slice the small and big jumps in a non-homogeneous way with respect to the characteristic scale of an $\alpha$-stable process $\jsize \mapsto \jsize^{\frac{1}{\alpha}}$.\\
	
	\noindent Pick $\jsize\geq t_0$. Using \eqref{eq:decomposition_stable}, the $\alpha$-stable symmetric Lévy driving process can be written as
	\[L_t-L_{t_0}=\int_{t_0}^t\int_{\abs{z}\leq \jsize^{\frac{1}{\alpha}}}z\widetilde{N}(\dd s, \dd z) +\int_{t_0}^t\int_{\abs{z}> \jsize^{\frac{1}{\alpha}}}z N(\dd s, \dd z). \] Indeed, the term $$ \int_{t_0}^t \int_{|z|> \xi^{\frac{1}{\alpha}}}z \nu(\dd z) \dd s$$ is equal to $0$ since $\nu$ is a symmetric measure. 
	\begin{steps}
		\item We first apply Itô's formula (see \citer[Theorem 4.4.7 p. 251]{ApplebaumLevyprocessesstochastic2009}) and estimate the expectation of each term for $\kappa\leq 1$, in order to get \eqref{eq: moment 1_alpha}.\\
	
		\noindent Fix $\eta>0$ to be chosen latter and define the $\mathcal{C}^2$-function $f:(x,v)\mapsto (\eta+x^2+v^2)^{\kappa/2}$. We use the fact that for all $y \in \mathbb{R}^2$, $y \cdot Ay = 0$, and observe that for any $s\geq t_0$ and $(x,v) \in \mathbb{R}^2$, $(x,v)^T\cdot F(s,v) =\abs{v}^{\gamma+1}s^{-\beta}\geq0$. In the sequel, we write, for simplicity, $f(Z_{s^-} + z)$ for $f(X_{s^-}, V_{s^-} +z)$. For all $t\geq t_0$, by Itô's formula, we have
		\begin{equation*}\label{eq:moment_Ito}
			f(Z_{t\wedge \tau_r})\leq  f(z_0)+M_{t}+R_{t}+S_{t},
		\end{equation*}
		where
		\begin{equation*}\label{eq: martingale}
			M_t:=\int_{t_0}^t \int_{0<\abs{z}<\jsize^{\frac{1}{\alpha}}} \mathbb{1}_{\{s\leq \tau_r\} } (f(Z_{s-}+z)-f(Z_{s-})) \widetilde{N}(\dd s, \dd z),
		\end{equation*}
		\begin{equation}\label{eq: jumps}
			R_t:=\int_{t_0}^t \int_{\abs{z}\geq \jsize^{\frac{1}{\alpha}}}\mathbb{1}_{\{s\leq \tau_r\} } (f(Z_{s-}+z)-f(Z_{s-})) N(\dd s, \dd z),
		\end{equation}
		\begin{equation}\label{eq: finite variation}
			S_t:=\int_{t_0}^t \int_{0<\abs{z}<\jsize^{\frac{1}{\alpha}}}\mathbb{1}_{\{s\leq \tau_r\} } \left[f(Z_{s-}+z)-f(Z_{s-})-\nabla f(Z_{s-}).z\right] \nu(\dd z) \dd s.
		\end{equation}
		Moreover, remark that for all $k>\alpha$,
		\begin{equation}\label{eq: levy_measure_petits}
			\int_{0<\abs{z}<\jsize^{\frac{1}{\alpha}}} \abs{z}^{k} \nu(\dd z ) = \frac{2a}{k-\alpha}\jsize^{\frac{k}{\alpha}-1},
		\end{equation}
		and for all $k<\alpha$,
		\begin{equation}\label{eq: levy_measure_ grands}
			\int_{\abs{z}\geq\jsize^{\frac{1}{\alpha}}} \abs{z}^{k} \nu(\dd z ) = \frac{2a}{\alpha-k}\jsize^{\frac{k}{\alpha}-1}.
		\end{equation}
		We estimate expectations of $M$, $R$ and $S$.\\
		To that end, we first show that the local martingale $(M_{t})_{t\geq t_0}$ is a martingale. Fix $q\geq 2$ and $r\geq0$.
		Moreover, we set
		\[ I_t(q):=\int_{t_0}^{t} \int_{0<\abs{z}<\jsize^{\frac{1}{\alpha}}}\mathbb{1}_{\{s\leq \tau_r\} } \abs{f(Z_{s-}+z)-f(Z_{s-})}^q \nu(\dd z )\dd s.\]  Thanks to Taylor-Lagrange inequality, for all $\absd{(x,v)}\leq r$ and $\abs{z}\leq \jsize^{\frac{1}{\alpha}}$, 
		\[\abs{f(x,v+z)-f(x,v)}\leq \sup\{\absd{\nabla f(y)},\ \absd{y}\in [0, r+\jsize^{\nicefrac{1}{\alpha}}] \}\abs{z}\leq C_{r,\jsize,\kappa}\abs{z},\]
		so we have
		\[ I_t(q)\leq C_{r,\jsize,\kappa}
		\int_{t_0}^{t} \int_{0<\abs{z}<\jsize^{\frac{1}{\alpha}}}\mathbb{1}_{\{s\leq \tau_r\} } \abs{z}^{q} \nu(\dd z )\dd s. \]
		Hence, it is a finite quantity, since $q\geq 2$ and \eqref{eq: levy_measure_petits} holds.
		Therefore, for $q\geq 2$, by Kunita's inequality (see \citer[Theorem 4.4.23 p. 265]{ApplebaumLevyprocessesstochastic2009}), there exists $D_q>0$ such that
		\[
		\begin{aligned}
			\mathbb{E}\left[\sup_{t_0\leq s \leq t} \abs{M_s}^q\right] \leq D_q \left(\mathbb{E}\left[ I_t(2)^{\frac{q}{2}} \right] +\mathbb{E}\left[ I_t(q)\right]\right) <+\infty.
		\end{aligned}
		\]
		Hence, by \citer[Theorem 51 p. 38]{ProtterStochasticIntegrationDifferential2005}, $M$ is a martingale.
		\newline
		We estimate now the finite variation part $S$ defined in \eqref{eq: finite variation}.
		Note that for all $(x,v)\in \mathbb{R}^2$, the Hessian matrix of $f$ is given by
		\[
		\mathrm{Hess}(f)(x,v)=\kappa (x^2+v^2+\eta)^{\frac{\kappa}{2}-1}\begin{pmatrix}
			1+(\kappa-2)\dfrac{x^2}{x^2+v^2+\eta} & (\kappa-2)\dfrac{xv}{x^2+v^2+\eta} \\(\kappa-2)\dfrac{xv}{x^2+v^2+\eta} & 1+(\kappa-2)\dfrac{v^2}{x^2+v^2+\eta}
		\end{pmatrix}.\]
		Its matrix norm is bounded by $C_{\kappa}\eta^{\frac{\kappa}{2}-1}$. \\
		Assume that $\abs{z}<\jsize^{\frac{1}{\alpha}}$.
		Using Taylor- Lagrange's inequality and injecting \eqref{eq: levy_measure_petits} we get the almost sure following bound, for all $s \geq t_0$,
		\begin{equation}\label{deriveeseconde}
			\abs{ \int_{0<\abs{z}<\jsize^{\frac{1}{\alpha}}} \left(f(Z_{s-}+z)-f(Z_{s-})-\nabla f(Z_{s-})\cdot z\right) \nu(\dd z)} \leq C_{\kappa}\eta ^{\frac{\kappa}{2}-1}\frac{2a}{2-\alpha}\jsize^{\frac{2}{\alpha}-1} .
		\end{equation}
		
		It remains to study the Poisson integral $R$ defined in \eqref{eq: jumps}. Recall that $\kappa\leq 1$. Let us note that for all $x,v,z \in \mathbb{R}$, by Hölder property of power functions, one has 
		
		\[\begin{aligned}
			\abs{f(x,v+z)-f(x,v)} & \leq \abs{\left(\eta +x^2+(v+z)^2\right)^{\frac{\kappa}{2}}-\left(x^2+(v+z)^2\right)^{\frac{\kappa}{2}}} 
			\\ &\quad +\abs{\left(x^2+(v+z)^2\right)^{\frac{\kappa}{2}}  - \left(x^2+v^2\right)^{\frac{\kappa}{2}} } 
			+\abs{\left(x^2+v^2\right)^{\frac{\kappa}{2}}-\left(\eta+x^2+v^2\right)^{\frac{\kappa}{2}}}\\& \leq 2\eta^{\frac{\kappa}{2}} + \abs{\absd{(x,v+z)}^{\kappa}-\absd{(x,v)}^{\kappa}} \\ &\leq 2\eta^{\frac{\kappa}{2}}+ \abs{z}^{\kappa}.
		\end{aligned}\]

		Injecting \eqref{eq: levy_measure_ grands}, we deduce that
		\begin{equation}\label{poissonintegral}
			\int_{\abs{z}\geq \jsize^{\frac{1}{\alpha}}}\abs{ f(Z_{s-}+z)-f(Z_{s-})} \nu(\dd z) \leq \eta^{\frac{\kappa}{2}} \frac{2a}{\alpha }\jsize^{-1}+\frac{2a}{\alpha-\kappa}\jsize^{\frac{\kappa}{\alpha}-1}.
		\end{equation}
		Moment estimate of the Poisson integral follows from \citer[Theorem 2.3.7 p. 106]{ApplebaumLevyprocessesstochastic2009}. \\
		Gathering \eqref{poissonintegral} and \eqref{deriveeseconde}, we obtain
		\begin{multline*}
			\mathbb{E}\left[\absd{Z_{t\wedge \tau_r}}^{\kappa}\right]\leq \mathbb{E}\left[f(Z_{t\wedge \tau_r})  \right] \leq \mathbb{E}\left[f(Z_{t_0})\right]+t \jsize^{-1}  \left(\eta^{\kappa/2} \frac{2a}{\alpha }+\frac{2a}{\alpha-\kappa}\jsize^{\frac{\kappa}{\alpha}}+ C_{\kappa}\eta^{\frac{\kappa}{2}-1} \frac{2a}{2-\alpha}\jsize^{\frac{2}{\alpha}}\right).
		\end{multline*}
		Choosing $\eta=t^{\frac{2}{\alpha}}$ and $\jsize=t$, we get
		\begin{equation*}\label{eq:moment_step1}
			\mathbb{E}\left[\absd{Z_{t\wedge \tau_r}}^{\kappa}\right] \leq \mathbb{E}\left[f(Z_{t_0})\right]+t^{\frac{\kappa}{\alpha}}  \left( \frac{2a}{\alpha }+\frac{2a}{\alpha-\kappa}+ C_{\kappa} \frac{2a}{2-\alpha}\right)\leq C_{\kappa, t_0}t^{\frac{\kappa}{\alpha}}.
		\end{equation*}
		Thanks to \cref{explosion}, we can conclude that the explosion time of $Z$ is a.s.\ infinite, and letting $r \to +\infty$ with Fatou's lemma, for all $\kappa\in [0,1]$,
		\begin{equation}\label{eq: kappa_1}
			\mathbb{E}\left[\absd{Z_{t}}^{\kappa}\right] \leq C_{\kappa, t_0}t^{\frac{\kappa}{\alpha}}.
		\end{equation}
		\item Pick $\kappa\in (1,\alpha)$. We estimate $R$ in another way, using again \citer[Theorem 2.3.7 p. 106]{ApplebaumLevyprocessesstochastic2009}.\\ 
		
		By the Hölder property of power function and \eqref{eq: levy_measure_ grands}, we get
		\begin{equation}\begin{aligned}\label{poissonintegralbis}
				\int_{\abs{z}\geq \jsize^{\frac{1}{\alpha}}} \abs{f(Z_{s-}+z)-f(Z_{s-})} \nu(\dd z) & \leq \int_{\abs{z}\geq \jsize^{\frac{1}{\alpha}}} \abs{2zV_{s-}+z^2}^{\frac{\kappa}{2}}\nu(\dd z) \\&\leq C_{\kappa}\left(\frac{2a}{\alpha-\kappa}\jsize^{\frac{\kappa}{\alpha}-1}+\abs{V_{s-}}^{\frac{\kappa}{2}}\frac{2a}{\alpha-\frac{\kappa}{2}}\jsize^{\frac{\kappa}{2\alpha}-1}\right).
			\end{aligned}
		\end{equation}
		Gathering \eqref{deriveeseconde} and \eqref{poissonintegralbis}, one has
		
		\begin{multline*}\label{eq: moment intermediaire}
			\mathbb{E}\left[\absd{Z_{t\wedge \tau_r}}^{\kappa}\right] \leq \mathbb{E}\left[f(Z_{t_0})\right]+t\left(C_{\kappa}\frac{2a}{\alpha-\kappa}\jsize^{\frac{\kappa}{\alpha}-1}+C_{\kappa}\eta^{\frac{\kappa}{2}-1} \frac{2a}{2-\alpha}\jsize^{\frac{2}{\alpha}-1}\right)\\+ C_{\kappa}\frac{2a}{\alpha-\frac{\kappa}{2}}\jsize^{\frac{\kappa}{2\alpha}-1}\int_{t_0}^{t}\mathbb{E}\left[\abs{V_s}^{\frac{\kappa}{2}} \right]\dd s.
		\end{multline*}
		Injecting \eqref{eq: kappa_1} applied with $\frac{\kappa}{2}$, choosing $\eta=t^{\frac{2}{\alpha}}$ and $\jsize=t$, we get
		\begin{equation*}
			\mathbb{E}\left[\absd{Z_{t\wedge \tau_r}}^{\kappa}\right] \leq C_{\kappa, t_0,\alpha}t^{\frac{\kappa}{\alpha}}.
		\end{equation*}
		
		The conclusion of the proof follows, letting $r\to +\infty$.
	\end{steps}
\end{proof}

\subsection{Asymptotic behavior of the solution}
We gather in this section the proof of \cref{ThmlongtimeY_stable}. The strategy is to prove the convergence of the f.d.d.\ of the process $Y^{(\eps)}$, and then its tightness both in the super-critical and critical regimes.
We first prove the tightness when $\alpha q \geq 1$. Recall that $q= \frac{\beta}{\gamma + \alpha -1}$.
\begin{lem}\label{lem: tight_levy}
	Assume that $\alpha q \geq 1$, then the family $\left\{(\eps^{\frac{1}{\alpha}}Y_{t/\eps})_{t\geq \eps t_0 },\ \eps >0\right\}$ is tight on every compact interval $[m,M]$, for $0<m\leq M$.
\end{lem}
\begin{proof}
	We check the Aldous's tightness criterion stated in \citer[Theorem 16.10 p. 178]{BillingsleyConvergenceprobabilitymeasures1999}. 
	Let $a,\ \eta,\ T$ be positive reals. Let $\tau$ be a discrete stopping time with finite range $\mathcal{T}$, bounded by $T$ and
	fix $\delta>0$ and $\eps>0$ small enough to be chosen later. Let us define the Wiener-Lévy integral appearing in \eqref{EDS_Y_Levy} $ M^{(\eps)}$ by, for all $t \geq \eps t_0$,
	\begin{align}\label{eq: local_martingale_stable}
		\notag M_t^{(\eps)} &:=\eps^{\frac{1}{\alpha}}\int_{t_0}^{t/\eps}e^{-sA}\Gamma\dd S_s \\\notag  &= \eps^{\frac{1}{\alpha}}\int_{t_0}^{t/\eps}\begin{pmatrix}
			- \sin(s) \\
			\cos(s)
		\end{pmatrix}\dd L_s \\ &=  \eps^{\frac{1}{\alpha}}\int_{t_0}^{t/\eps} \int_{\mathbb{R}} \begin{pmatrix}
			- \sin(s) \\
			\cos(s)
		\end{pmatrix}z \widetilde{N}(\dd s, \dd z),
	\end{align}
	using the representation \eqref{eq:decomposition_stable} of $L$. Notice that it is a martingale. By the triangle inequality, one has  
	\begin{equation}\label{eq_decomp}\mathbb{E}\left[\absd{Y_{\tau+\delta}^{(\eps)}-Y_{\tau}^{(\eps)}} \right]\leq \mathbb{E}\left[\absd{M^{(\eps)}_{\tau +\delta} - M^{(\eps)}_{\tau}} \right] + \mathbb{E}\left[\eps^{\frac{1}{\alpha}}\int_{\tau/\eps}^{(\tau+\delta)/\eps}\abs{V_u}^{\gamma}u^{-\beta}\dd u \right]. \end{equation} 
	Writing $M^{(\eps)}=: \left( M^{(\eps),1}, M^{(\eps),2}  \right)^T,$ the quadratic variations of $M^{(\eps),1}$ and $ M^{(\eps),2}$ satisfy, by \cite{ApplebaumLevyprocessesstochastic2009} (see $(4.15)$ p. $257$), for all $t \geq \eps t_0$ \begin{align*}\text{Tr}\left(\left[M^{(\eps)}_t, M^{(\eps)}_t\right]\right) &:= \left[ M^{(\eps),1}_t,  M^{(\eps),1}_t\right] +  \left[ M^{(\eps),2}_t,  M^{(\eps),2}_t\right]\\ & = \int_{t_0}^{t/\eps} \int_{\mathbb{R}} \absd{\eps^{\frac{1}{\alpha}}  \begin{pmatrix}
				- \sin(u) \\
				\cos(u)
			\end{pmatrix}z}^2 \, N(\dd u, \dd z).\end{align*}
	Using Burkholder-Davis-Gundi's inequality (see \citer[Theorem $48$ p. $193$]{ProtterStochasticIntegrationDifferential2005}), we deduce that for some constant $C>0$ independent of $\tau, \delta$ and $\eps$ which may change from line to line \begin{align*}
		\mathbb{E}\left[\absd{M^{(\eps)}_{\tau +\delta} - M^{(\eps)}_{\tau}} \right] & \leq C \mathbb{E} \left[\left(\text{Tr}\left(\left[M^{(\eps)}_{\tau +\delta}, M^{(\eps)}_{\tau + \delta}\right] - \left[M^{(\eps)}_{\tau }, M^{(\eps)}_{\tau}\right] \right)\right)^{\frac{1}{2}}\right] \\ &\leq C \mathbb{E} \left[ \left( \int_{\tau / \eps}^{(\tau + \delta)/\eps} \int_{\mathbb{R}} |\eps^{\frac{1}{\alpha}}z|^2 \, N(\dd u, \dd z)\right)^{\frac{1}{2}}\right] \\ &= C  \mathbb{E} \left[\left(\left[L^{(\eps)}_{\tau +\delta}, L^{(\eps)}_{\tau + \delta}\right] - \left[L^{(\eps)}_{\tau }, L^{(\eps)}_{\tau}\right] \right)^{\frac{1}{2}}\right],
	\end{align*}
	where $(L^{(\eps)}_t)_{t\geq 0} := (\eps^{\frac{1}{\alpha}} L_{t/\eps})_{t\geq0}$ has the same distribution as $L$ by its self-similarity property. Using this and the lower-bound in Burkholder-Davis-Gundi's inequality, we deduce that 
	
	\begin{align*}
		\mathbb{E}\left[\absd{M^{(\eps)}_{\tau +\delta} - M^{(\eps)}_{\tau}} \right]  & \leq C  \mathbb{E} \left[ \sup_{0 \leq s \leq \delta} |L_{\tau + s} - L_{\tau} | \right] \\ &=   C  \mathbb{E} \left[ \sup_{0 \leq s \leq \delta} |L_{ s} | \right] \\ &=  C  \mathbb{E} \left[ \sup_{0 \leq s \leq \delta} \delta^{\frac{1}{\alpha}}|L_{s/\delta} | \right]  \\ &=  C  \delta^{\frac{1}{\alpha}}\mathbb{E} \left[ \sup_{0 \leq s \leq 1} |L_{ s} | \right],
	\end{align*}
	where the first equality stems from the strong Markov property satisfied by $L$ and the second one from the self-similarity property of $L$ again. Note that $\mathbb{E} \left[ \sup_{0 \leq s \leq 1} |L_{ s} | \right]$ is finite thanks to \cite{luschgy:hal-00085213} (see Section $3$) since $\alpha >1$.\\
	
	Since $\tau \in [m,M]$ a.s., the last term in \eqref{eq_decomp} can be handled as in \eqref{major_espectation} using moment estimates of $V$ (see \cref{esperance2}). It yields
	\[ \mathbb{E}\left[\eps^{\frac{1}{\alpha}}\int_{\tau/\eps}^{(\tau+\delta)/\eps}\abs{V_u}^{\gamma}u^{-\beta}\dd u \right] \leq C_{m,M}\eps^{\beta-\frac{\gamma+\alpha-1}{\alpha}}.\]
	
	Since $\eta >0$ and by Markov's inequality, we obtain for $\delta$ and $\eps$ small enough
	\[ \mathbb{P}\left(\absd{Y_{\tau+\delta}^{(\eps)}- Y_{\tau}^{(\eps)}}\geq a  \right) \leq \dfrac{C\delta^{\frac{1}{\alpha}}+C_{m,M}\eps^{\beta-\frac{\gamma+\alpha-1}{\alpha}}}{a} \leq \eta.
	\]
	Moreover, by Markov's inequality and the moment estimates again, we deduce that for all $t\in [m,M]$,
	\[\lim_{a\to +\infty}\limsup_{\eps\to0}\mathbb{P}\left(\absd{Y_t^{(\eps)}}\geq a \right)\leq \lim_{a\to +\infty}\limsup_{\eps\to0}\dfrac{\mathbb{E}\left[ \absd{Y_t^{(\eps)}}\right]}{a}\leq \lim_{a\to +\infty}\dfrac{Ct^{\frac{1}{\alpha}}}{a}=0. \]
	By \citer[Corollary and Theorem 16.8 p. 175]{BillingsleyConvergenceprobabilitymeasures1999}, this concludes the proof of the tightness on every compact interval of $(0,+\infty)$.
\end{proof}
We will now prove the convergence of the f.d.d.\ of $Y^{(\eps)}$. Thanks to the previous lemma, this will yield the weak convergence on every compact set (see \citer[Theorem 13.1 p. 139]{BillingsleyConvergenceprobabilitymeasures1999}) in the super-critical and critical regimes. The convergence in distribution on the whole space $\mathcal{D}$ will follow from \citer[Theorem 16.7 p. 174]{BillingsleyConvergenceprobabilitymeasures1999}, since all processes considered are càdlàg.
\subsubsection{Convergence of the f.d.d.\ in the super-critical regime}
Assume that $\alpha q >1$. Recall that $(Y_t^{(\eps)})_{t\geq\eps t_0}:=(\eps^{\frac{1}{\alpha}}Y_{t/\eps})_{t\geq \eps t_0}$.
\begin{proof}[Proof of \cref{ThmlongtimeY_stable} \textit{(i)}]~ \begin{steps}
		\item  We first prove the convergence of the f.d.d.\ of the Wiener-Lévy integral appearing in \eqref{EDS_Y_Levy}.\\
		Recall that the local martingale $M^{(\eps)}$ was defined in \eqref{eq: local_martingale_stable}.
		\item[Step 1a.]
		We begin with the convergence in distribution of $ M_{s,t}^{(\eps)}:=M_t^{(\eps)} - M_s^{(\eps)}$, for $\eps t_0\leq s \leq t$.\\ To this end, we study the characteristic function $\phi^{(\eps)}_{s,t}$ of $M_{s,t}^{(\eps)}$. Let us recall that $\psi$ denotes the characteristic exponent of $L$, and is given, for all $\xi \in \mathbb{R}$, by 
		\[\psi(\xi) = -a|\xi|^\alpha.\] 
		The characteristic function of the Wiener-Lévy integral can be computed as \citer[p. 105]{SatoLevyprocessesinfinitely1999}, hence one has, for all $\xi:=(u,v) \in \mathbb{R}^2$, \begin{align*}
			\phi^{(\eps)}_{s,t}(\xi) & =\mathbb{E} \left( \exp\left[-iu \eps^{\frac{1}{\alpha}}
			\int_{s/\eps}^{t/\eps} \sin(y)\dd L_y + iv \eps^{\frac{1}{\alpha}}
			\int_{s/\eps}^{t/\eps}\cos(y)\dd L_y\right]\right)                                  \\ &=  \mathbb{E} \left( \exp\left[i \eps^{\frac{1}{\alpha}}
			\int_{s/\eps}^{t/\eps}(-u\sin(y) + v \cos(y))\dd L_y\right]\right)                  \\ &= \exp\left( \int_{s/\eps}^{t/\eps} \psi \left( \eps^{\frac{1}{\alpha}}[-u\sin(y) + v \cos(y)]\right) \dd y \right)\\ &= \exp\left( -a \eps \int_{s/\eps}^{t/\eps} |-u\sin(y) + v \cos(y)|^\alpha \dd y \right).\end{align*}
		Using \cref{lemma_average_periodic2}, we deduce that $\phi^{(\eps)}_{s,t}(\xi)$ converges, as $\eps\to 0$, to 
		\[	\exp \left( -a(t-s) \frac{1}{2\pi}\int_0^{2\pi}\abs{-u\sin(y) + v \cos(y)}^\alpha   \dd y\right). \]
		
	\item[Step 1b.] We now compute explicitly the scale parameter of the stable limiting process.\\
		We denote by $\lambda$ the uniform probability distribution on the circle $\mathds{S}^1$. Thanks to a change of variable and the symmetry of $\lambda$, setting $\omega := \frac{\xi}{\Vert \xi \Vert}$ for $\xi = (u,v) \in \mathbb{R}^2\setminus\{0\}$,  we have
		\begin{align*}
			\frac{1}{2\pi}\int_0^{2\pi}\abs{-u\sin(y) + v \cos(y)}^\alpha  \dd y   &= \int_{\mathds{S}^1} \left\vert  \xi\cdot \lambda  \right\vert^\alpha \, \dd \lambda \\ &=\Vert\xi\Vert^\alpha \int_{\mathds{S}^1} \left\vert \omega\cdot \lambda  \right\vert^\alpha \, \dd \lambda. \end{align*}
		Since $\lambda$ is rotationally invariant, we deduce that $\int_{\mathds{S}^1} \left\vert  \omega \cdot \lambda \right\vert^\alpha \, \dd \lambda$ does not depend on $\omega \in \mathds{S}^1$. Taking $\omega = (1,0)^T$, we set
		\begin{equation}\label{constantCtilde}\widetilde{C}:= \frac{a}{2\pi}\int_{0}^{2\pi} \abs{\cos(x)}^{\alpha}\dd x.\end{equation}
		We have thus proved that, for any $\xi \in \mathbb{R}^2$,
		\[	\phi^{(\eps)}_{s,t}(\xi)\quad \underset{\eps \to 0}{\longrightarrow}\quad \exp\left( -(t-s) \widetilde{C} \Vert \xi \Vert^\alpha \right).\]
		
		Thus, the following convergence in distribution holds 
		\begin{equation}\label{proofLevycvMeps}
			M_{s,t}^{(\eps)} =M_t^{(\eps)} - M_s^{(\eps)}\quad \underset{\eps \to 0}{\Longrightarrow} \quad\mathcal{L}_{t-s}.
		\end{equation}
		Following the same lines, we show that, for any $t>0$,  
		\begin{equation}\label{proofLevycvMeps2}
			M_{t}^{(\eps)} \quad \underset{\eps \to 0}{\Longrightarrow} \quad\mathcal{L}_{t}.
		\end{equation}
		\item[Step 1c.]	We now prove the convergence in f.d.d.\ of $M^{(\eps)}$ to $\mathcal{L}$, as $\eps$ tends to $0$.\\ Let us fix $0 < t_1 \leq t_2 \leq \cdots \leq t_d$. Note that $(M_t^{(\eps)})_{t\geq \eps t_0}$ is a càdlàg process with independent increments, since the integrands in its definition are deterministic and because $L$ is a Lévy process. Thus, the random variables $(M_{t_1}^{(\eps)}, M_{t_1,t_2}^{(\eps)}, \dots, M_{t_{d-1},t_d}^{(\eps)} )$ are mutually independent. We deduce from the convergence results established in \eqref{proofLevycvMeps} and \eqref{proofLevycvMeps2}, and the fact that $\mathcal{L}$ has stationary and independent increments that 
		\[ (M_{t_1}^{(\eps)}, M_{t_1,t_2}^{(\eps)}, \dots, M_{t_{d-1},t_d}^{(\eps)} ) \quad \underset{\eps \to 0}{\Longrightarrow} \quad ( \mathcal{L}_{t_1}, \mathcal{L}_{t_2} - \mathcal{L}_{t_1}, \dots, \mathcal{L}_{t_d} - \mathcal{L}_{t_{d-1}}).\]
		The continuous mapping theorem yields the convergence in f.d.d.\ of $M^{(\eps)}$ to $\mathcal{L}$.
		\item  Pick $T> 0$. We prove that 
		\begin{equation*}
			\mathbb{E}\left[\sup_{\eps t_0\leq t \leq T}\absd{Y_t^{(\eps)}- M_t^{(\eps)}}\right]\quad \underset{\eps \to 0}{\longrightarrow} \quad 0.
		\end{equation*}
		Let us fix $\eps>0$ small enough such that $\eps t_0 \leq T$. We have
		\[\sup_{\eps t_0\leq t \leq T}\absd{Y_t^{(\eps)}-M_t^{(\eps)}} \leq \eps^{\frac{1}{\alpha}}\absd{z_0} +\eps^{\frac{1}{\alpha}}\int_{t_0}^{T/\eps}\absd{e^{-sA}F(s,V_s)}\dd s.  \]
		We use moment estimates (\cref{esperance2}) to get
		\[\begin{aligned}
			\mathbb{E}\left[\eps^{\frac{1}{\alpha}}\int_{t_0}^{T/\eps}\absd{e^{-sA}F(s,V_s)}\dd s\right] & =\mathbb{E}\left[\eps^{\frac{1}{\alpha}}\int_{t_0}^{T/\eps}\absd{F(s,V_s)}\dd s\right] \\ & \leq \mathbb{E}\left[\eps^{\frac{1}{\alpha}}\int_{t_0}^{T/\eps}\abs{V_s}^{\gamma}s^{-\beta}\dd s\right] \\ & \leq \eps^{\frac{1}{\alpha}}C_{\gamma,t_0}\int_{t_0}^{T/\eps}s^{\frac{\gamma}{\alpha}-\beta}\dd s \\ & \leq C_{\gamma,t_0} (\eps^{\beta-\frac{\gamma+\alpha -1}{\alpha}}T^{\frac{\gamma}{\alpha}-\beta+1}+\eps^{\frac{1}{\alpha}}t_0^{\frac{\gamma}{\alpha}-\beta+1}).
		\end{aligned} \]
		Hence, setting $r:= \min(\beta-\frac{\gamma+ \alpha -1}{\alpha}, \frac{1}{\alpha} )$, which is positive by assumption, we get
		\begin{equation}
			\mathbb{E}\left[\sup_{\eps t_0\leq t \leq T}\absd{Y_t^{(\eps)}- M_t^{(\eps)}}\right] = \underset{\eps \to 0}{O}(\eps^r).
		\end{equation}
		The conclusion follows from \citer[Theorem 3.1 p. 27]{BillingsleyConvergenceprobabilitymeasures1999}.
	\end{steps}
\end{proof}
\subsubsection{Convergence of the f.d.d.\ in the critical and sub-critical regime}
In this section, we consider the linear case, i.e.\ $\gamma=1$ and we assume that $\beta \in \left(\frac{1}{2},1\right]$. Recall that 
\[(Y_t^{(\eps)})_{t\geq \eps t_0}= \left(\eps^{q}Y_{t/\eps}\right)_{t\geq \eps t_0}.\]
\begin{proof}[Proof of \cref{ThmlongtimeY_stable} $(ii)$ and $(iii)$]
	The proof follows the same lines as in the Brownian setting. Leaving out the noise, recall that the underlying ODE is the following
	\begin{equation}\label{eq: ode_sscritiq2}
		x''(t)+\dfrac{x'(t)}{t^{\beta}}+x(t)=0, \quad t \geq t_0.
	\end{equation}
	We pick again the basis of solutions given by \cref{eq: resolvante }, and we still denote by $R$ its resolvent matrix and by $f$ its rate of decrease.
	Recall that it is given by 
	\begin{equation}\label{defrateofdecrease2}
		\forall t >0, \, f(t):= \begin{cases}
			\frac{1}{\sqrt{ t}}                               & \mbox{if }\beta=1, \\
			\exp \left(-\frac{t^{1-\beta}}{2(1-\beta)}\right) & \mbox{else.}
		\end{cases}
	\end{equation}
	We set, for all $t\geq \eps t_0$,
	\begin{equation}\label{eq: widetildeM_levy}
		\widetilde{M}^{(\eps)}_t:= \eps^q f(t/\eps)\int_{t_0}^{t/\eps}R_s^{-1}\Gamma \dd S_s.
	\end{equation}
	Keeping the same notations as in the Brownian case, we decompose $(Y_t)_{t\geq t_0}= (e^{-tA}Z_t)_{t\geq t_0}$ into
	\begin{equation*}
		\eps^qY_{t/\eps}= \eps^qf(t/\eps)\Phi_{t/\eps}R_{t_0}^{-1}Z_0+\Phi_{t/\eps}\widetilde{M}^{(\eps)}_t.
	\end{equation*}
	Reasoning as in the Brownian case, it remains to study the convergence of $\widetilde{M}^{(\eps)}$ since the first term converges towards $0$.
	Using the expression of the Wronskian obtained in \cref{eq: resolvante }, we obtain, for all $t\geq \eps t_0$,
	\begin{equation*}
		\widetilde{M}^{(\eps)}_t = \eps^qf(t/\eps)\int_{t_0}^{t/\eps}f(u)^{-2}\begin{pmatrix}
			-y_2(u) \\ y_1(u)
		\end{pmatrix} \dd L_u.
	\end{equation*}
	Let us fix $0<s<t$. We study the convergence in distribution of the couple $(\widetilde{M}^{(\eps)}_s, \widetilde{M}^{(\eps)}_t)$ when $\eps$ tends to $0$. The convergence in distribution of a general $d$-dimensional distribution $(\widetilde{M}^{(\eps)}_{t_1},\dots,\widetilde{M}^{(\eps)}_{t_d} )$ relies on the same computations.\\ Let us fix $(\xi_1,\xi_2) \in \mathbb{R}^2\times\mathbb{R}^2$. Using that $L$ has independent increments, the characteristic function $\widetilde{\phi}^{(\eps)}_{s,t}$ of $(\widetilde{M}^{(\eps)}_s, \widetilde{M}^{(\eps)}_t)$ is given by
	\begin{align}\label{eq_caract_function_eps}
		\notag\widetilde{\phi}^{(\eps)}_{s,t}(\xi_1,\xi_2) &= \mathbb{E} \left[\exp\left( i\eps^q \left[f\left(s/\eps\right) \xi_1 \cdot \int_{t_0}^{s/\eps} f(u)^{-2}\begin{pmatrix}
			-y_2(u) \\ y_1(u)
		\end{pmatrix} \dd L_u \right.\right.\right. \\  \notag &\hspace{5cm}\left.\left.\left. + f\left(t/\eps\right) \xi_2 \cdot \int_{t_0}^{t/\eps} f(u)^{-2}\begin{pmatrix}
			-y_2(u) \\ y_1(u)
		\end{pmatrix} \dd L_u \right]\right)\right] \\ &= \mathbb{E} \left[\exp\left( i\eps^q \left(f\left(s/\eps\right) \xi_1 + f\left(t/\eps\right) \xi_2 \right) \cdot \int_{t_0}^{s/\eps} f(u)^{-2}\begin{pmatrix}
			-y_2(u) \\ y_1(u)
		\end{pmatrix} \dd L_u \right)\right] \\ &\notag \hspace{4cm} \times  \mathbb{E} \left[\exp\left( i\eps^q f\left(t/\eps\right) \xi_2 \cdot \int_{s/\eps}^{t/\eps} f(u)^{-2}\begin{pmatrix}
			-y_2(u) \\ y_1(u)
		\end{pmatrix} \dd L_u \right)\right].
	\end{align}
	Let us recall that the characteristic exponent of $L$ is given, for all $\xi \in \mathbb{R}$, by 
	\[\psi(\xi) = -a|\xi|^\alpha.\]
	The characteristic function of the Wiener-Lévy integral can be computed as \citer[p. 105]{SatoLevyprocessesinfinitely1999}. Indeed, if $G: \mathbb{R} \to \mathbb{R}$ is a continuous function, then we have for all $z \in \mathbb{R}$ $$ \mathbb{E}\left[ \exp \left(iz \int_{s}^t G(u)  \dd L_u\right)\right] = \exp\left( -a \int_{s}^t |G(u)z|^\alpha \dd u\right).$$ Using this with $z=1$ and $$G : u \in [t_0,+\infty) \mapsto \eps^q \left(f\left(s/\eps\right) \xi_1 + f\left(t/\eps\right) \xi_2 \right) \cdot \left[ f(u)^{-2}\begin{pmatrix}
		-y_2(u) \\ y_1(u)
	\end{pmatrix} \right]$$ to compute the first expectation in \eqref{eq_caract_function_eps} and the corresponding function $G$ for the second expectation, one has \begin{multline*}
		\widetilde{\phi}^{(\eps)}_{s,t}(\xi_1,\xi_2) = \exp\left( -a\eps^{\beta}   \int_{t_0}^{s/\eps} f(u)^{-2\alpha} \left\vert \left(f\left(s/\eps\right) \xi_1 + f\left(t/\eps\right) \xi_2 \right)\cdot  \begin{pmatrix}
			-y_2(u) \\ y_1(u)
		\end{pmatrix}  \right\vert^{\alpha} \dd u \right) \\  \times \exp\left( -a\eps^{\beta} \int_{s/\eps}^{t/\eps} f(u)^{-2\alpha} \left\vert  f\left(t/\eps\right) \xi_2 \cdot  \begin{pmatrix}
			-y_2(u) \\ y_1(u)
		\end{pmatrix}  \right\vert^{\alpha} \dd u \right).
	\end{multline*}
	Using the asymptotic expansion of the resolvent matrix (\cref{eq: resolvante }), we can write, for any $u \geq t_0$, 
	\[\begin{pmatrix}
		-y_2(u) \\ y_1(u)
	\end{pmatrix} = f(u) \left[\begin{pmatrix}
		-\sin(u) \\ \cos(u)
	\end{pmatrix} + g(u)\right], \] 
	where $g: [t_0,+\infty) \to \mathbb{R}^2$ is a function satisfying for all $u \geq t_0$,
	\[|g(u)| \leq C u^{1-2\beta}.\]
	We set
	\begin{equation}\label{eq: chractfunct stable critical}
		K^{(\eps)}_1:= \exp\left( -a\eps^{\beta} \int_{t_0}^{s/\eps} f(u)^{-\alpha} \left\vert \left(f\left(s/\eps\right) \xi_1 + f\left(t/\eps\right) \xi_2 \right)\cdot \left[\begin{pmatrix}
			-\sin(u) \\ \cos(u)
		\end{pmatrix} + g(u) \right]  \right\vert^{\alpha} \dd u \right)
	\end{equation}
	and
	\begin{equation*}
		K^{(\eps)}_2:=\exp\left( -a\eps^{\beta} \int_{s/\eps}^{t/\eps} f(u)^{-\alpha} \left\vert  f\left(\frac{t}{\eps}\right) \xi_2 \cdot \left[\begin{pmatrix}
			-\sin(u) \\ \cos(u)
		\end{pmatrix} + g(u) \right] \right\vert^{\alpha} \dd u \right).
	\end{equation*}
	We thus obtain
	\begin{equation}
		\widetilde{\phi}^{(\eps)}_{s,t}(\xi_1,\xi_2)=K^{(\eps)}_1 \times K^{(\eps)}_2,
	\end{equation}

	\begin{steps}
		\item  We start by justifying that we can omit $g$ to study the limit when $\eps \to 0$. More precisely, we prove that, for all
		function $\zeta: \mathbb{R}\to \mathbb{R}^2$ such that $\absd{\zeta(\eps)}f(s/\eps)^{-1} = \underset{\eps \to 0}{O}(1)$,
		\begin{equation}\label{eq: proofstablecritical1}R^{(\eps)}:=
			\eps^{\beta} \int_{t_0}^{s/\eps} f(u)^{-\alpha}\abs{\abs{ \zeta(\eps)\cdot \left[\begin{pmatrix}
						-\sin(u) \\ \cos(u)
					\end{pmatrix} + g(u) \right]  }^{\alpha}- \abs{\zeta(\eps)\cdot \left[\begin{pmatrix}
						-\sin(u) \\ \cos(u)
					\end{pmatrix} \right]  }^{\alpha} }\dd u \underset{\eps \to 0}{\longrightarrow} 0.
		\end{equation}
	
		Thanks to the mean value theorem applied to $\lvert \,\cdot \,\rvert^{\alpha}$ (since $\alpha \geq1$), and the domination of $g$, we obtain that, for some constant $C>0$,
		\begin{equation*}
			R^{(\eps)}\leq  C\eps^{\beta} \absd{\zeta(\eps)}^{\alpha}\int_{t_0}^{s/\eps} f(u)^{-\alpha} u^{1-2\beta} \dd u=\underset{\eps \to 0}{O}(\absd{\zeta(\eps)}^{\alpha}f(s/\eps)^{-\alpha}\eps^{2\beta -1}),
		\end{equation*}
		where the last equality follows from \cref{lemmaintegralasymptoticexpansion}. This proves \eqref{eq: proofstablecritical1} since $\beta > \frac12$.
		\item We focus on the first term $K^{(\eps)}_1$ defined in \eqref{eq: chractfunct stable critical}.
		Since $f$ is decreasing, notice that 
		\[\zeta(\eps):=f\left(s/\eps\right) \xi_1 + f\left(t/\eps\right) \xi_2 = \underset{\eps \to 0}{O}(f(s/\eps)).\] 
		Then we have to study the convergence of $I^{(\eps)}$ defined by
		\begin{equation*}
			I^{(\eps)}:= a\eps^{\beta} \int_{t_0}^{s/\eps} f(u)^{-\alpha} \left\vert \left(f\left(s/\eps\right) \xi_1 + f\left(t/\eps\right) \xi_2 \right)\cdot \begin{pmatrix}
				-\sin(u) \\ \cos(u)
			\end{pmatrix}   \right\vert^{\alpha} \dd u.
		\end{equation*}
		Its limit differs according to the value of $\beta$.
		\item[Step 2a.] Assume first that $\beta=1$. Then, using the expression of $f$ (see \eqref{defrateofdecrease2}),
		\begin{equation*}
			I^{(\eps)}=a\eps^{1+\frac{\alpha}{2}} \int_{t_0}^{s/\eps} f(u)^{-\alpha} \left\vert \left(\frac{\xi_1}{\sqrt{s}}  + \frac{\xi_2}{\sqrt{t}} \right)\cdot \begin{pmatrix}
				-\sin(u) \\ \cos(u)
			\end{pmatrix}   \right\vert^{\alpha} \dd u.
		\end{equation*}
		We proved in Step 1B of the super-critical regime that there exists a constant $\widetilde{C}>0$ given in \eqref{constantCtilde} such that, for all $\zeta \in \mathbb{R}^2$,
		\begin{equation}\label{eq: widetildeC}
			\frac{a}{2\pi}\int_{0}^{2\pi} \left\vert  \zeta\cdot \begin{pmatrix}
				-\sin(u) \\ \cos(u)
			\end{pmatrix}  \right\vert^{\alpha} \dd u = \widetilde{C}\Vert \zeta\Vert^{\alpha}.
		\end{equation}
		Using \cref{lemma_average_periodic2}, we can compute the following asymptotic expansion
		\begin{equation*}
			I^{(\eps)}=\eps^{1+\frac{\alpha}{2}}\widetilde{C}\absd{\frac{\xi_1}{\sqrt{s}}  + \frac{\xi_2}{\sqrt{t}}}^{\alpha}\int_{t_0}^{s/\eps} f(u)^{-\alpha} \dd u + \underset{\eps \to 0}{o} \left(\eps^{1+\frac{\alpha}{2}} \int_{t_0}^{s/\eps} f(u)^{-\alpha} \dd u\right).
		\end{equation*}
		Therefore, it follows from \cref{lemmaintegralasymptoticexpansion} that
		\begin{equation*}\label{eq: Ieps_beta_egal1}
			K^{(\eps)}_1\quad \underset{\eps \to 0}{\longrightarrow}\quad  \exp\left(-\widetilde{C}\left(1+\frac{\alpha}{2} \right)^{-1}\absd{\frac{\xi_1}{\sqrt{s}}  + \frac{\xi_2}{\sqrt{t}}}^{\alpha} s^{1+\frac{\alpha}{2}}\right).
		\end{equation*}
		\item[Step 2b.] Let us consider now $\beta \in \left(\frac{1}{2},1\right)$. Let us notice that $I^{(\eps)}$ can be decomposed into the sum
		\begin{equation}\label{eq: proofstablecritical2}
			I^{(\eps)}= I^{(\eps)}_1 + I^{(\eps)}_2
		\end{equation} of the two following terms
		\begin{equation*}
			I^{(\eps)}_1:=a\eps^{\beta} \int_{t_0}^{s/\eps} f(u)^{-\alpha} \left\vert f\left(\frac{s}{\eps}\right) \xi_1 \cdot \begin{pmatrix}
				-\sin(u) \\ \cos(u)
			\end{pmatrix}  \right\vert^{\alpha} \dd u
		\end{equation*}
		and
		\begin{equation*}
			I^{(\eps)}_2:= a\eps^{\beta} \int_{t_0}^{s/\eps} f(u)^{-\alpha} \left[\left\vert \left(f\left(s/\eps\right) \xi_1 + f\left(t/\eps\right) \xi_2 \right)\cdot \begin{pmatrix}
				-\sin(u) \\ \cos(u)
			\end{pmatrix}   \right\vert^{\alpha} - \left\vert f\left(s/\eps\right) \xi_1 \cdot \begin{pmatrix}
				-\sin(u) \\ \cos(u)
			\end{pmatrix}  \right\vert^{\alpha} \right] \dd u.
		\end{equation*}
		
		Using again the mean value theorem and \cref{lemmaintegralasymptoticexpansion}, we get that for some positive constant $C_{\Vert \xi_1\Vert,\Vert \xi_2\Vert}$,
		\begin{equation}\label{eq: proofstablecritical3}
			|I^{(\eps)}_2| \leq  C \eps^{\beta}f\left(s/\eps\right)^{\alpha-1}  f\left(t/\eps\right)  \int_{t_0}^{s/\eps} f(u)^{-\alpha}  \dd u = \underset{\eps \to 0}{O} \left(f(t/\eps)f(s/\eps)^{-1} \right) =\underset{\eps \to 0}{o}(1),
		\end{equation}
		since $\beta<1$. Using \cref{lemma_average_periodic2}, we can compute the following asymptotic expansion of $I^{(\eps)}_1$
		\begin{equation*}
			\begin{aligned}
				I^{(\eps)}_1 & = a\eps^{\beta} f\left(\frac{s}{\eps}\right)^{\alpha} \int_{t_0}^{s/\eps} f(u)^{-\alpha} \left\vert  \xi_1 \cdot \begin{pmatrix}
					-\sin(u) \\ \cos(u)
				\end{pmatrix}  \right\vert^{\alpha} \dd u \\ &=  a\eps^{\beta} f\left(\frac{s}{\eps}\right)^{\alpha} \left[ \left(\frac{1}{2\pi}\int_{0}^{2\pi} \left\vert  \xi_1 \cdot \begin{pmatrix}
					-\sin(u) \\ \cos(u)
				\end{pmatrix}  \right\vert^{\alpha} \dd u\right) \int_{t_0}^{s/\eps} f(u)^{-\alpha} \dd u + \underset{\eps \to 0}{o} \left( \int_{t_0}^{s/\eps} f(u)^{-\alpha} \dd u\right)\right].
			\end{aligned}
		\end{equation*}
		Thanks to \eqref{eq: widetildeC} and the asymptotic expansion's results given in \cref{lemmaintegralasymptoticexpansion}, there exists an explicit constant $k_{\beta,\alpha}$ given in \cref{lemmaintegralasymptoticexpansion}, such that
		\begin{equation}\label{eq: proofstablecritical4}
			I^{(\eps)}_1 \quad \underset{\eps \to 0}{\longrightarrow}\quad  k_{\beta,\alpha}\widetilde{C}s^{\beta} \Vert \xi_1 \Vert^{\alpha}.
		\end{equation}
		Combining \eqref{eq: proofstablecritical1},  \eqref{eq: proofstablecritical2}, \eqref{eq: proofstablecritical3} and \eqref{eq: proofstablecritical4}, we have proved that $K^{(\eps)}_1$, defined in \eqref{eq: chractfunct stable critical}, converges as $\eps \to 0$ towards 
		\[ \exp\left(-k_{\beta,\alpha}\widetilde{C} s^{\beta} \Vert \xi_1 \Vert^{\alpha}\right).\]
		\item It remains to deal with the limit of $K^{(\eps)}_2$.
		Notice that 
		\[\zeta(\eps):= f\left(t/\eps\right)\xi_2= \underset{\eps \to 0}{O}\left(f\left(t/\eps\right)\right)=\underset{\eps \to 0}{O}\left(f\left(s/\eps\right)\right).\] 
		Hence, thanks to Step 1, we are reduced to study, for $r\in \{s, t \}$,
		\begin{equation}
			J^{(\eps)}_r:= a\eps^{\beta}f\left(\frac{t}{\eps}\right)^{\alpha}\int_{t_0}^{r/\eps}f(u)^{-\alpha} \left\vert  \xi_2 \cdot \begin{pmatrix}
				-\sin(u) \\ \cos(u)
			\end{pmatrix} \right\vert^{\alpha} \dd u.
		\end{equation}
		Asymptotic expansion's results (Lemmas \ref{lemma_average_periodic2} and \ref{lemmaintegralasymptoticexpansion}) and \eqref{eq: widetildeC} yield
		\begin{equation*}
			\begin{aligned}
				J_r^{(\eps)} & = \widetilde{C}\absd{\xi_2}^{\alpha}\eps^{\beta}f\left(t/\eps\right)^{\alpha}\int_{t_0}^{r/\eps}f(u)^{-\alpha}\dd u + \underset{\eps \to 0}{o}\left(f\left(t/\eps\right)^{\alpha}f\left(r/\eps\right)^{-\alpha}\right) \\ &= \widetilde{C}\absd{\xi_2}^{\alpha}k_{\beta,\alpha}r^{\beta}f\left(t/\eps\right)^{\alpha}f\left(r/\eps\right)^{-\alpha}+ \underset{\eps \to 0}{o}\left(f\left(t/\eps\right)^{\alpha}f\left(r/\eps\right)^{-\alpha}\right).
			\end{aligned}
		\end{equation*}
		Hence,
		\begin{equation}
			J_t^{(\eps)}\quad \underset{\eps \to 0}{\longrightarrow}\quad
			\widetilde{C}\absd{\xi_2}^{\alpha}k_{\beta,\alpha}t^{\beta}
		\end{equation}
		and
		\begin{equation}
			J_s^{(\eps)}\quad \underset{\eps \to 0}{\longrightarrow}\quad
			\widetilde{C}\absd{\xi_2}^{\alpha}k_{\beta,\alpha}s^{\beta}\left(\frac{s}{t}\right)^{\frac{\alpha}{2}}\mathbb{1}_{\{\beta= 1\}}.
		\end{equation}
		
		Since 
		\[K^{(\eps)}_2= \exp\left(-J_t^{(\eps)}+J_s^{(\eps)}\right),\]
		we thus obtain that, for all $0 <s \leq t$,
		\[ \widetilde{\phi}^{(\eps)}_{s,t}(\xi_1,\xi_2)\quad \underset{\eps \to 0}{\longrightarrow}\quad  \begin{cases}
			\exp\left(-k_{\beta,\alpha}\widetilde{C} s^{\beta} \Vert \xi_1 \Vert^{\alpha}\right) \exp\left(-k_{\beta,\alpha}\widetilde{C} t^{\beta} \Vert \xi_2 \Vert^{\alpha}\right)                                                                  & \mbox{if }\beta<1,  \\
			\exp\left(-k_{\beta,\alpha}\widetilde{C} \left[\absd{\frac{\xi_1}{\sqrt{s}}+\frac{\xi_2}{\sqrt{t}}}^{\alpha}s^{1+\frac{\alpha}{2}}+\absd{\xi_2}^{\alpha}t-\absd{\xi_2}^{\alpha}\left(\frac{s}{t}\right)^{\frac{\alpha}{2}}s \right]\right) & \mbox{if } \beta=1.
		\end{cases}\]
		\item We can compute the characteristic function of the process $\left(\frac{1}{\sqrt{t}}\int_{0}^t\sqrt{s}\dd \mathcal{L}_{s}\right)_{t>0}$ in the same manner, and thus recognize the limiting process in the critical regime.
	\end{steps}

\end{proof}
\begin{remark}
	\begin{itemize}
		\item As in the Brownian setting, if $\beta =0$, the resolvent matrix is explicit and following the same lines, we can prove that $\left(Z_{t/\eps}\right)_{t\geq \eps t_0}$ converges in f.d.d.\ towards the product of the measure $\mu$, whose characteristic function is given by
		\[\xi \mapsto \exp\left(-a \int_0^{+\infty}e^{-\alpha u}\absd{\xi\cdot h(u)
		}^{\alpha}\dd u\right),\]
		$h$ being an explicit periodic function depending on the resolvent matrix.
		
		\item As in the Brownian setting, since the asymptotic expansion of the resolvant matrix is also known in the super-critical regime, i.e. $\beta>1$, one can prove the result in the linear case, i.e. $\gamma=1$, following the same lines.  
	\end{itemize}

\end{remark}
\subsection{Proof of \cref{Corollary_cv_Z}}

\begin{proof}[Proof of \cref{Corollary_cv_Z}]
	We start by proving the convergence in distribution of $r_{1/T}Z_T$. Reasoning as in the Brownian setting, it follows from \cref{ThmlongtimeY_stable} that $r_{1/T} Y_T$ converges. The conclusion is a consequence of \cref{Lemma_oscillating_gaussian_convergence}, noting that the limiting distribution is invariant under rotations thanks to the expression of its characteristic function. \\
	Let us now prove that the rescaled process $Z^{(\eps)}$ does not converge in distribution. We state the proof in the super-critical regime. Assume by contradiction that it is the case. Reasoning as in the Brownian case, we prove that this implies the convergence in distribution of the process $I^{(\eps)}$ defined, for $t \geq \eps t_0$, by
	
	\[ I_t^{(\eps)}:= \eps^{\frac{1}{\alpha}}
	\int_{t_0}^{t/\eps} \sin\left( \frac{t}{\eps} - u \right)\dd L_u.\]
	In particular, for $s< t$, the random variable $I_t^{(\eps)} - I_s^{(\eps)}$ shall converge in distribution. \\
	
	\noindent Let us denote by $\phi^{(\eps)}$ the characteristic function of $I_t^{(\eps)} - I_s^{(\eps)}$, which is supposed to converge on $\mathbb{R}$. Using that $L$ has independent increments, we have   \begin{align*}
		\phi^{(\eps)}(1) & = \mathbb{E }\left[ \exp \left( \eps^{\frac{1}{\alpha}}\int_{s / \eps}^{t/\eps} \sin \left( \frac{t}{\eps} - u \right) \, \dd L_u \right)\right]   \mathbb{E }\left[ \exp \left( \eps^{\frac{1}{\alpha}}\int_{t_0}^{s/\eps} \sin \left( \frac{t}{\eps} - u \right) - \sin \left( \frac{s}{\eps} - u \right) \, \dd L_u \right)\right] \\ &=:\phi^{(\eps),1}\phi^{(\eps),2}.
	\end{align*}
	Recall that $\psi$ defined in \eqref{eq:symbol_stable}, denotes the characteristic exponent of $L$. Using a change of variables, we have in particular \begin{align*}
		\phi^{(\eps),1} & = \exp \left( \int_{s / \eps}^{t/\eps} \psi\left(\eps^{\frac{1}{\alpha}}\sin \left( \frac{t}{\eps} - u \right)\right)  \dd u \right) = \exp \left( -a\eps\int_{s / \eps}^{t/\eps} \left\vert\sin \left( \frac{t}{\eps} - u \right)\right\vert^{\alpha} \dd u \right) \\ &= \exp \left( -a\eps\int_{0}^{(t-s)/\eps} \left\vert\sin \left(u \right)\right\vert^{\alpha} \dd u \right).
	\end{align*}
	
	\cref{lemma_average_periodic2} ensures that $\phi^{(\eps),1}$ has a limit when $\eps$ converges to $0$. Similarly, we obtain
	\begin{align}\label{proofnocvZstableeq1}
		\notag\phi^{(\eps),2} & = \exp \left( \int_{t_0}^{s/\eps} \psi \left(\eps^{\frac{1}{\alpha}}\sin \left( \frac{t}{\eps} - u \right) - \sin \left( \frac{s}{\eps} - u \right)\right) \dd u  \right) \\  \notag&= \exp \left( -a\eps\int_{t_0}^{s/\eps} \left\vert\sin \left( \frac{t}{\eps} - u \right) - \sin \left( \frac{s}{\eps} - u \right)\right\vert^{\alpha} \dd u \right) \\ \notag &= \exp \left( -a2^\alpha \left|\sin \left( \frac{t-s}{2\eps}\right)\right\vert^\alpha \eps\int_{t_0}^{s/\eps} \left\vert\cos \left( \frac{t+s}{2\eps} - u \right) \right\vert^{\alpha} \dd u \right) \\  &= \exp \left( -a2^\alpha \left|\sin \left( \frac{t-s}{2\eps}\right)\right\vert^\alpha \eps\int_{\frac{t-s}{2\eps}}^{\frac{t+s}{2\eps} - t_0} \left\vert\cos \left(  u \right) \right\vert^{\alpha}  \dd u \right).
	\end{align}
	The change of variables $ u = v + \pi$ yields, for all $\eps>0$, 
	\[ \int_0^{2\pi} \left\vert\cos \left(  u \right) \right\vert^{\alpha} \sgn\left(\sin \left( \frac{t-s}{2\eps}\right) \cos \left( u \right)\right)  \dd u  = 0.\]
	Thus, \cref{lemma_average_periodic2} ensures that 
	\[\eps\int_{\frac{t-s}{2\eps}}^{\frac{t+s}{2\eps} - t_0} \left\vert\cos \left(  u \right) \right\vert^{\alpha}  \dd u \quad \underset{\eps \to 0}{\longrightarrow} \quad \frac{s}{2\pi} \int_0^{2\pi} \vert \cos(u)\vert^\alpha \dd u. \]
	
	Coming back to \eqref{proofnocvZstableeq1}, we see that $\phi^{(\eps),2}$ does not converge when $\eps$ tends to $0$. This is a contradiction.
\end{proof}
\appendix
\section{Study of the deterministic underlying ODE}\label{appendix_ode}
The deterministic ODE behind the system, i.e. without frictional force and without noise, is the following
\begin{equation}\label{eq: ode_souscritiq}
	x''(t)+\dfrac{x'(t)}{t^{\beta}}+x(t)=0, \quad t \geq t_0.
\end{equation}
The solutions form a vector space of dimension 2. Let us take two solutions $y_1$ and $y_2$ which are linearly independent. Then, we introduce the fundamental system of solutions (resolvent matrix) $R$ to \eqref{eq: ode_souscritiq} defined, for $t\geq t_0$, by
\begin{equation*}
	R_t=\begin{pmatrix}
		y_1(t) & y_2(t) \\ y_1'(t)& y_2'(t)
	\end{pmatrix}.
\end{equation*}
It satisfies, for all $t\geq t_0$,
\begin{equation*}
	R_t'= \begin{pmatrix}
		0 & 1 \\ -1& -\frac{1}{t^{\beta}}
	\end{pmatrix}R_t.
\end{equation*}
We recall that the Wronskian $w$ is defined, for all $t \geq t_0$, by 
\[w(t)=y_1(t)y_2'(t) - y_1'(t)y_2(t).\]
Let us finally set, for $t >0$,
\begin{equation}\label{eq: rateofdecreaseresolvent}
	f(t):= \begin{cases}
		\frac{1}{\sqrt{ t}}                               & \mbox{if }\beta=1, \\
		\exp \left(-\frac{t^{1-\beta}}{2(1-\beta)}\right) & \mbox{else.}
	\end{cases}
\end{equation}

\begin{lem}\label{expansion_bessel2} Pick $\beta \in \left(\frac{1}{2}, +\infty\right)$ and consider a solution $y$ to \eqref{eq: ode_souscritiq}. Then, there exist $a\in \mathbb{R}$ and $\phi \in [0,2\pi)$ such that
	\begin{equation*}
		y(t)=af(t)\cos(t+\phi)+\underset{t\to \infty}{O}\left(f(t)t^{-(2\beta-1)\wedge \beta}\right)
	\end{equation*}
	and
	\begin{equation*}
		y'(t)=-af(t)\sin(t+\phi)+\underset{t\to \infty}{O}\left(f(t)t^{-(2\beta-1)\wedge \beta}\right).
	\end{equation*}
\end{lem}
\begin{proof}
	Let us set, for $t\geq t_0$, $u(t)=f(t)^{-1}y(t)$. We easily check that $u$ satisfies
	\[ u''(t) + u(t) \left[ 1 + h(t)\right]=0,\]
	where $h(t):= \frac{f''(t)}{f(t)} + \frac{f'(t)}{f(t)t^{\beta}}$. Since for all $t \geq t_0$ $$ \frac{f'(t)}{f(t)} = - \frac{1}{2t^\beta},$$ we obtain that $ h(t)= \underset{t\to + \infty}{O}(t^{-(2\beta)\wedge (\beta+1)})$. Following the proof of the method of variation
	of parameters, there exists $a_0,b_0 \in \mathbb{R}$ such that, for any $t \geq t_0$,
	\[ u(t) = a_0 \cos(t) + b_0 \sin(t) - \int_{t_0}^t u(s) h(s) \sin(t-s) \, \dd s.\]
	Using that $h \in L^1((t_0,+\infty))$ since $\beta > \frac{1}{2},$ we obtain by Grönwall's lemma that $u$ is bounded on $[t_0, +\infty)$. We deduce that the functions $s \mapsto u(s)h(s)\cos(s)$ and $s \mapsto u(s) h(s) \sin(s)$ belong to $L^1((t_0, + \infty))$.  Thus, up to changing the constants $a_0$ and $b_0,$ one has, for all $t \geq t_0$,
	\begin{equation}\label{eq: devp_u_proof_edo} u(t) = a_0 \cos(t) + b_0 \sin(t) - \sin(t)\int_{t}^{\infty} u(s) h(s) \cos(s) \, \dd s + \cos(t) \int_{t}^{\infty} u(s) h(s) \sin(s) \, \dd s.\end{equation}
	It follows from the fact that $u$ is bounded that
	\[ u(t) = a_0 \cos(t) + b_0 \sin(t) + \underset{t \to + \infty}{O}\left( \int_{t}^{\infty} \frac{\dd s}{s^{(2\beta) \wedge (\beta+1)}} \right).  \]
	Thus, there exist $a \in \mathbb{R}$ and $\phi \in [0,2\pi)$ such that
	\[ u(t) = a\cos(t+\phi) + \underset{t \to + \infty}{O}(t^{-(2\beta-1)\wedge \beta}). \]
	This proves the asymptotic expansion of $y$. Differentiating \eqref{eq: devp_u_proof_edo} and using that $h(t) = \underset{t\to + \infty}{O}(t^{-(2\beta)\wedge (\beta+1)})$, we prove that
	\[ u'(t) = -a\sin(t+\phi) +  \underset{t \to + \infty}{O}(t^{-(2\beta-1)\wedge \beta}). \]
	Since $u$ is bounded and $f'(t) =\underset{t \to + \infty}{O}(f(t)t^{-\beta}),$ we finally obtain that
	\[y'(t)= f'(t)u(t) + f(t)u'(t) = f(t) u'(t) + \underset{t \to + \infty}{O}(f(t)t^{-\beta}). \]
	This concludes the proof of the asymptotic expansion of $y'$.
\end{proof}
\begin{remark}
	Note that if $\beta=1$, the Bessel functions of the first kind $J_0$ and of the second kind $Y_0$ form a basis of solutions. Their asymptotic expansions can be found in \cite[Chap VII]{WatsontreatisetheoryBessel1944}.
\end{remark}
\begin{lem}\label{eq: resolvante }Pick $\beta \in \left(\frac{1}{2}, +\infty\right)$. There exists a basis of solutions $y_1$ and $y_2$ to \eqref{eq: ode_souscritiq} such that the resolvent matrix $R$ satisfies
	\begin{equation*}
		R_t= f(t)\begin{pmatrix}
			\cos(t) & \sin(t) \\ -\sin(t) & \cos(t)
		\end{pmatrix} +\underset{t\to \infty}{O}\left(f(t)t^{-(2\beta-1)\wedge \beta}\right)=f(t) e^{tA}+\underset{t\to \infty}{O}\left(f(t)t^{-(2\beta-1)\wedge \beta}\right).
	\end{equation*}
	Moreover, its Wronskian $w$ is given for any $t \geq t_0$ by 
	\[w(t) = f(t)^2.\]
\end{lem}
\begin{proof}
	
	It is well-known that the Wronskian satisfies, for all $t \geq t_0$, 
	\[ w'(t) = -\frac{1}{t^{\beta}} w(t).\]
	Thus, there exists $w_0 \in \mathbb{R}\setminus\{0\}$ such that, for all $t \geq t_0$, $w(t)=w_0f(t)^2$. Moreover, thanks to \cref{expansion_bessel2}, for $i\in \{1,2\}$, there exist $a_i\in \mathbb{R}$ and $\phi_i\in [0,2\pi)$ such that
	\begin{equation*}
		y_i(t)=a_if(t)\cos(t+\phi_i)+\underset{t\to \infty}{O}\left(f(t)t^{-(2\beta-1)\wedge \beta}\right)
	\end{equation*}
	and
	\begin{equation*}
		y_i'(t)=-a_if(t)\sin(t+\phi_i)+\underset{t\to \infty}{O}\left(f(t)t^{-(2\beta-1)\wedge \beta}\right).
	\end{equation*}
	As a consequence,
	\begin{equation*}
		w(t)= -a_1a_2f(t)^2 \sin (\phi_2-\phi_1)+\underset{t\to \infty}{O}\left(f(t)^2t^{-(4\beta-2)\wedge 2\beta}\right).
	\end{equation*}
	But since $w(t)=w_0f(t)^2$, it implies that $a_i\neq 0$ and $\phi_2 \not\equiv \phi_1 [\pi]$.\\
	Up to dividing by $a_i$, we can assume that $a_i=1$, and up to considering a linear combination of $y_1$ and $y_2$, we can assume that $\phi_1=0$ and $\phi_2= -\frac{\pi}{2}$. Thus, we have $w_0=1$. This concludes the proof.
\end{proof}

\section{Some technical results}\label{appendix}
We collect here some technical results used in our proofs.
Recall first a sufficient condition for the non-explosion of the solution to a SDE. The proof can be found in \cite{GradinaruAsymptoticbehaviortimeinhomogeneous2021a}.
\begin{lem}\label{explosion} Let $(Y_t)_{t\geq t_0}$ be a càdlàg process, solution to a SDE. For all $n\geq0$, define the stopping time
	\begin{equation}
		\tau_n:=\inf \{ t\geq t_0, \ \absd{Y_t}\geq n\}.
	\end{equation}
	Set $\tau_{\infty}:= \lim_{n\to +\infty}\tau_n$ the explosion time of $Y$.
	Assume that there exist two measurable and non-negative functions $\phi$ and $b$ such that
	\begin{enumerate}[label=(\roman*)]
		\item $\phi$ is non-decreasing and $\lim_{n\to \infty}\phi(n)=+\infty$,
		\item $b$ is finite-valued,
		\item and for all $t\geq t_0$,
		\begin{equation*}\label{borne_explosion}
			\sup_{n\geq 0}\mathbb{E}\left[\phi(\abs{Y_{t\wedge \tau_{n}}})\right] \leq b(t).
		\end{equation*}
	\end{enumerate}
	Then $\tau_{\infty}=+\infty$ a.s.
\end{lem}
We now state and prove a result on the periodic-averaging phenomenon.
\begin{lem}\label{lemma_average_periodic2} Let us fix $t_0>0$ and $h: [t_0,+\infty) \to \mathbb{R}$ a continuous $m$-periodic function, with $m >0$. Let $g : [t_0,+\infty) \to \mathbb{R}^+$ be a continuously differentiable function which is not integrable on $[t_0,+\infty)$. We assume moreover that \begin{enumerate}[label=(\roman*)]
		\item $g'(t) =\underset{t \to + \infty}{o}(g(t))$,
		\item $g(t)=\underset{t \to + \infty}{o}\left(\int_{t_0}^{t} g(u) \dd u \right)$.
	\end{enumerate}
	
	Then, 
	\[ \int_{t_0}^{t}g(u) h(u)  \dd u  = \left[\frac{1}{m}  \int_{t_0}^{t_0+m} h(u) \dd u \right]\int_{t_0}^{t}g(u)\dd u + \underset{t \to + \infty}{o} \left( \int_{t_0}^{t}g(u)\dd u \right).\]
	Let us remark that the functions $g_1$ and $g_2$ defined for $t \in \mathbb{R}$, $r \geq 0$ and $\beta \in [0,1)$ by $g_1(t):= t^r$ and $g_2(t):= \exp(rt^{1-\beta})$, satisfy the preceding assumptions made on $g$.
\end{lem}

\begin{proof}
	Let us define $\widetilde{h}:= h - \frac{1}{m}  \int_{t_0}^{t_0+m} h(u) \dd u $, and $\widetilde{H}$ a primitive of $\widetilde{h}$. The function $\widetilde{H}$ is bounded on $[t_0,+\infty)$ since the average of $\widetilde{h}$ on its period is equal to $0$. To prove the lemma, we only need to justify that 
	\[\int_{t_0}^{t}g(u) \widetilde{h}(u)  \dd u = \underset{t \to + \infty}{o}\left(\int_{t_0}^{t} g(u) \dd u \right).\] 
	By integration by parts, we obtain that, for all $t \geq t_0$,
	\[ \int_{t_0}^{t}g(u) \widetilde{h}(u)  \dd u = g(t)\widetilde{H}(t) - g(t_0)\widetilde{H}(t_0) - \int_{t_0}^t g'(u) \widetilde{H}(u) \dd u.\]
	Using the fact that $\widetilde{H}$ is bounded, that $g'(t) = \underset{t \to + \infty}{o}(g(t))$ and that $\int_{t_0}^{\infty} g(u) \dd u = + \infty$, we deduce that 
	\[ \int_{t_0}^t g'(u) \widetilde{H}(u) \dd u = \underset{t \to + \infty}{o}\left(\int_{t_0}^{t} g(u) \dd u \right).\] 
	The conclusion follows from the fact that $ g(t)\widetilde{H}(t) - g(t_0)\widetilde{H}(t_0) = \underset{t \to +\infty}{o}\left(\int_{t_0}^{t} g(u) \dd u \right)$, since we have assumed that $g(t)=\underset{t \to + \infty}{o}\left(\int_{t_0}^{t} g(u) \dd u \right)$ and that $\int_{t_0}^{\infty} g(u) \dd u = + \infty$.
	
\end{proof}

\begin{lem} \label{lemmaintegralasymptoticexpansion}
	Let $f$ be given by \eqref{eq: rateofdecreaseresolvent} for $\beta \in \left( \frac12 ,1 \right]$, and pick $\alpha \in (1,2]$. Define 	
	\begin{equation*}
		k_{\beta, \alpha}:=\begin{cases}
			(1+\alpha/2)^{-1} & \mbox{ if } \beta =1, \\
			\frac{2}{\alpha}  & \mbox{ else.}
		\end{cases}
	\end{equation*} 
	Then for any $t>0$, we have 
	\[ \int_{t_0}^{t/\eps} f(u)^{-\alpha} u^{1-2\beta} \dd u = \underset{\eps \to 0}{O}(f(t/\eps)^{-\alpha}\eps^{\beta -1}),\] 
	and  
	\[ \int_{t_0}^{t/\eps} f(u)^{-\alpha} \dd u = k_{\beta, \alpha} f\left(\frac{t}{\eps}\right)^{-\alpha} \left(\frac{t}{\eps}\right)^{\beta} + \underset{\eps \to 0}{o}(f(t/\eps)^{-\alpha}\eps^{-\beta}).\]
\end{lem}

\begin{proof}
	When $\beta = 1$, the results follow from direct computations because of the expression of $f$. Assume now that $\beta \in \left( \frac12,1 \right)$. For the first point, the integration by parts formula ensures that \begin{align*}
		\int_{t_0}^{t/\eps} f(u)^{-\alpha} u^{1-2\beta} \dd u & = \frac{2}{\alpha} \left[ f(u)^{-\alpha} u^{1-\beta} \right]^{t/\eps}_{t_0} - \frac{2}{\alpha}(1- \beta) \int_{t_0}^{t/\eps} f(u)^{-\alpha} u^{-\beta} \dd u \\ &= \underset{\eps \to 0}{O}(f(t/\eps)^{-\alpha}\eps^{\beta -1}) + \underset{\eps \to 0}{O}(f(t/\eps)^{-\alpha}) \\ &= \underset{\eps \to 0}{O}(f(t/\eps)^{-\alpha}\eps^{\beta -1}).
	\end{align*}
	For the second asymptotic expansion, it follows again from an integration by parts that \begin{align*}
		\int_{t_0}^{t/\eps} f(u)^{-\alpha}  \dd u & = \int_{t_0}^{t/\eps} f(u)^{-\alpha}u^{-\beta} u^{\beta}  \dd u \\ &= \frac{2}{\alpha} \left[ f(u)^{-\alpha} u^{\beta} \right]^{t/\eps}_{t_0} - \frac{2}{\alpha} \beta \int_{t_0}^{t/\eps} f(u)^{-\alpha} u^{\beta-1} \dd u.
	\end{align*}
	Remarking that $f(u)^{-\alpha} u^{\beta-1} = \underset{u \to + \infty}{o}(f(u)^{-\alpha})$, since $\beta  <1$, we deduce that \begin{align*}\int_{t_0}^{t/\eps} f(u)^{-\alpha} u^{\beta-1} \dd u &= \underset{\eps \to 0}{o} \left( \int_{t_0}^{t/\eps} f(u)^{-\alpha}\dd u\right). \end{align*} We obtain that 
	\[ \int_{t_0}^{t/\eps} f(u)^{-\alpha}\dd u \underset{\eps \to 0}{\sim} \frac2\alpha f\left(\frac{t}{\eps}\right)^{-\alpha} \left(\frac{t}{\eps}\right)^{\beta}  \] 
	This ends the proof.
\end{proof}
\begin{lem}\label{Lemma_oscillating_gaussian_convergence}
	Let $(X_n)_n$ be a sequence of random variables with values in $\mathbb{R}^2$, and which converges in distribution to a random variable $X$. We assume that the distribution of $X$ is invariant under rotations, i.e.\ for any orthogonal matrix $R \in \mathcal{M}_2(\mathbb{R})$, the random variables $X$ and $RX$ have the same distribution. Then for all sequence $(R_n)_n$ of orthogonal matrices in $\mathcal{M}_2(\mathbb{R})$, we have 
	\[ R_n X_n \quad \underset{n\to +\infty}{\Longrightarrow}\quad X.\]
\end{lem}

\begin{proof}
	Let us denote by $\phi_Z$  the characteristic function of a random variable $Z$.  Using \citer[Theorem 5.3 p. 86]{KallenbergFoundationsModernProbability2002}, we know that $(\phi_{X_n})_n$ converges to $\phi_X$ uniformly on every compact subset of $\mathbb{R}^2$.
	The characteristic function of the random variable $Y_n:=R_nX_n$ is given by 
	\[ \xi \mapsto \phi_{Y_n}(\xi)=\phi_{X_n}(R_n\xi).\]
	Thus, by assumption, we have, for all $\xi \in \mathbb{R}^2$, 
	\[ \phi_{X}(R_n\xi) = \phi_{X}(\xi).\] 
	It follows that, for any $\xi \in \mathbb{R}^2$ and $n \geq 0$,
	\[\begin{aligned}\left\vert \phi_{Y_n}(\xi) - \phi(\xi) \right\vert = \left\vert\phi_{X_n}(R_n \xi) - \phi_X(R_n\xi) \right\vert \leq \sup_{z \in \mathbb{R}^2,  \Vert z \Vert= \Vert \xi \Vert} \left\vert \phi_{X_n}(z) - \phi_X(z) \right\vert, \end{aligned}\] 
	which converges to $0$, as $n \to + \infty$. This ends the proof of the lemma.
\end{proof} 

{\bf Acknowledgements}
The authors would like to thank Paul-Eric Chaudru de Raynal and Mihai Gradinaru for their supervision and advices.
The authors would also like to thank Jürgen Angst for having suggested this subject and useful discussions.
\bibliographystyle{abbrv}
\bibliography{article_confined}

\begin{thebibliography}{10}

\bibitem{AlbeverioLongtimebehavior1994}
S.~Albeverio and A.~Klar.
\newblock Long time behavior of nonlinear stochastic oscillators: {{The}}
  one-dimensional {{Hamiltonian}} case.
\newblock {\em J. Math. Phys.}, 35(8):4005--4027, Aug. 1994.

\bibitem{ApplebaumLevyprocessesstochastic2009}
D.~Applebaum.
\newblock {\em Levy Processes and Stochastic Calculus}.
\newblock Cambridge {{Studies}} in {{Advanced Mathematics}}. {Cambridge
  University Press}, second edition, 2009.

\bibitem{BillingsleyConvergenceprobabilitymeasures1999}
P.~Billingsley.
\newblock {\em Convergence of Probability Measures}.
\newblock Wiley Series in Probability and Statistics. {{Probability}} and
  Statistics Section. {Wiley}, 2nd ed edition, 1999.

\bibitem{RaynalStrongexistenceuniqueness2017}
P.~E.~C. de~Raynal.
\newblock Strong existence and uniqueness for degenerate {{SDE}} with
  {{H\"older}} drift.
\newblock {\em Annales de l'Institut Henri Poincar\'e, Probabilit\'es et
  Statistiques}, 53(1):259--286, Feb. 2017.

\bibitem{DitlevsenObservationastablenoise1999}
P.~D. Ditlevsen.
\newblock Observation of {$\alpha$}-stable noise induced millennial climate
  changes from an ice-core record.
\newblock {\em Geophysical Research Letters}, 26:1441--1444, 1999.

\bibitem{FedrizziRegularitystochastickinetic2017}
E.~Fedrizzi, F.~Flandoli, E.~Priola, and J.~Vovelle.
\newblock Regularity of stochastic kinetic equations.
\newblock {\em Electron. J. Probab.}, 22, 2017.

\bibitem{FournierOnedimensionalcritical2021}
N.~Fournier and C.~Tardif.
\newblock One dimensional critical {{Kinetic Fokker-Planck}} equations,
  {{Bessel}} and stable processes.
\newblock {\em Communications in Mathematical Physics}, 381, Jan. 2021.

\bibitem{GradinaruAsymptoticbehaviortimeinhomogeneous2021a}
M.~Gradinaru and E.~Luirard.
\newblock Asymptotic behavior for a time-inhomogeneous {{Kolmogorov}} type
  diffusion, July 2021.

\bibitem{GradinaruAsymptoticbehaviortimeinhomogeneous2021}
M.~Gradinaru and E.~Luirard.
\newblock Asymptotic behavior for a time-inhomogeneous stochastic differential
  equation driven by an {$\alpha$}-stable {{L\'evy}} process, Dec. 2021.

\bibitem{GradinaruExistenceasymptoticbehaviour2013}
M.~Gradinaru and Y.~Offret.
\newblock Existence and asymptotic behaviour of some time-inhomogeneous
  diffusions.
\newblock {\em Ann. Inst. H. Poincar\'e Probab. Statist.}, 49(1):182--207, Feb.
  2013.

\bibitem{HonoreStrongregularizationBrownian2018}
I.~Honore, S.~Menozzi, and P.-E. {Chaudru de Raynal}.
\newblock Strong regularization by {{Brownian}} noise propagating through a
  weak {{H\"ormander}} structure, Oct. 2018.

\bibitem{Jourdain_fractional_diffusion}
B.~Jourdain, S.~Méléard, and W.~A. Woyczynski.
\newblock {A Probabilistic Approach for Nonlinear Equations Involving the
  Fractional Laplacian and a Singular Operator}.
\newblock {\em Potential Analysis}, 23:55--81, 2005.

\bibitem{KallenbergFoundationsModernProbability2002}
O.~Kallenberg.
\newblock {\em Foundations of {{Modern Probability}}}.
\newblock Probability and {{Its Applications}}. {Springer New York}, {New York,
  NY}, 2002.

\bibitem{KaratzasBrownianMotionStochastic1998}
I.~Karatzas and S.~Shreve.
\newblock {\em Brownian {{Motion}} and {{Stochastic Calculus}}}.
\newblock Graduate {{Texts}} in {{Mathematics}}. {Springer-Verlag}, {New York},
  second edition, 1998.

\bibitem{Langevintheoriemouvementbrownien1908}
P.~Langevin.
\newblock {Sur la th\'eorie du mouvement brownien}.
\newblock {\em Comptes-Rendus de l'Acad\'emie des Sciences}, 146:530--532, Jan.
  1908.

\bibitem{luschgy:hal-00085213}
H.~Luschgy and G.~Pag{\`e}s.
\newblock {Moment estimates for L{\'e}vy Processes}.
\newblock {\em {Electronic Journal of Probability}}, 13:422--434, 2008.

\bibitem{Mann_fractal_fractional_diffusions}
J.~A. Mann~Jr. and W.~A. Woyczynski.
\newblock {Growing fractal interfaces in the presence of self-similar hopping
  surface diffusion}.
\newblock {\em Physica A: Statistical Mechanics and its Applications},
  291:159--183, 2001.

\bibitem{MAO2011147}
X.~Mao.
\newblock {\em Stochastic Differential Equations and Applications}.
\newblock Woodhead Publishing, second edition edition, 2011.

\bibitem{MarinoWeakwellposednessdegenerate2021}
L.~Marino and S.~Menozzi.
\newblock Weak well-posedness for degenerate {{SDEs}} driven by {{L\'evy}}
  processes, Dec. 2021.

\bibitem{ProtterStochasticIntegrationDifferential2005}
P.~E. Protter.
\newblock {\em Stochastic Integration and Differential Equations}.
\newblock Sto\-chastic Modelling and Applied Probability. Springer-Verlag,
  Berlin Heidelberg, second edition, 2005.

\bibitem{SatoLevyprocessesinfinitely1999}
K.-i. Sato.
\newblock {\em L\'evy Processes and Infinitely Divisible Distributions}.
\newblock Number~68 in Cambridge Studies in Advanced Mathematics. {Cambridge
  University Press}, {Cambridge, U.K. ; New York}, 1999.

\bibitem{SituTheoryStochasticDifferential2005}
R.~Situ.
\newblock {\em Theory of {{Stochastic Differential Equations}} with {{Jumps}}
  and {{Applications}}: {{Mathematical}} and {{Analytical Techniques}} with
  {{Applications}} to {{Engineering}}}.
\newblock Mathematical and {{Analytical Techniques}} with {{Applications}} to
  {{Engineering}}. {Springer US}, 2005.

\bibitem{WangDegenerateSDEsHilbert2015}
F.-Y. Wang and X.~Zhang.
\newblock Degenerate {{SDEs}} in {{Hilbert Spaces}} with {{Rough Drifts}}, Jan.
  2015.

\bibitem{WatsontreatisetheoryBessel1944}
G.~N. Watson.
\newblock {\em A Treatise on the Theory of {{Bessel}} Functions}.
\newblock {Cambrige University Press}, second edition, 1944.

\bibitem{WuLargemoderatedeviations2001}
L.~Wu.
\newblock Large and moderate deviations and exponential convergence for
  stochastic damping {{Hamiltonian}} systems.
\newblock {\em Stochastic Processes and their Applications}, 91(2):205--238,
  Feb. 2001.

\bibitem{ZhangDensitiesSDEsDriven2014}
X.~Zhang.
\newblock Densities for sdes driven by degenerate $\alpha$-stable processes.
\newblock {\em The Annals of Probability}, 42(5):1885--1910, Sept. 2014.

\bibitem{ZhangStochasticHamiltonianflows2016}
X.~Zhang.
\newblock Stochastic {{Hamiltonian}} flows with singular coefficients, June
  2016.

\end{thebibliography}
\end{document}